\documentclass[11pt]{amsart}
\usepackage[margin=30mm]{geometry}
\usepackage[T1]{fontenc}
\usepackage{lmodern}
\usepackage{amsmath,amssymb,amsthm}
\usepackage{microtype}
\usepackage{setspace}
\usepackage[hidelinks]{hyperref}
\newtheorem{theorem}{Theorem}[section]
\newtheorem{lemma}[theorem]{Lemma}
\newtheorem{proposition}[theorem]{Proposition}

\theoremstyle{definition}
\newtheorem{definition}[theorem]{Definition}
\theoremstyle{remark}

\numberwithin{equation}{section}
\newcommand{\R}{\mathbb R}
\newcommand{\HH}{\mathcal H}
\newcommand{\LL}{\mathcal L}
\newcommand{\cof}{\operatorname{cof}}
\newcommand{\tr}{\operatorname{tr}}
\newcommand{\epi}{\operatorname{epi}}
\newcommand{\diag}{\operatorname{diag}}
\newcommand{\cc}{C_{\mathrm{cov}}}

\title[The singular set in optimal transport] {Hausdorff dimension of the singular set in optimal transport}

\author{Shibing Chen}
\address{}
\email{chenshibing1982@hotmail.com}

\author{Qi-Rui Li}
\address{School of Mathematical Sciences, Zhejiang University, Hangzhou 310058, China}
\email{qi-rui.li@zju.edu.cn}

\author{Yuanyuan Li}
\address{Institute for Theoretical Sciences, Westlake University, Hangzhou, 310030, China}
\email{lyyuan@westlake.edu.cn}

\author{Jiakun Liu}
\address{School of Mathematics and Statistics, The University of Sydney, Camperdown, NSW 2006, Australia}
\email{jiakun.liu@sydney.edu.au}

\author{Xiaotian Wu}
\address{Institute for Theoretical Sciences, Westlake University, Hangzhou, 310030, China}
\email{wuxiaotian@westlake.edu.cn}
\date{\today}
\subjclass[2020]{35J96, 49Q22}
\keywords{Optimal transport, Monge-Amp\`ere equation, singular set, Hausdorff dimension}

\begin{document}

\begin{abstract}
We study optimal transport between bounded smooth domains in $\R^n$, with smooth positive densities bounded above and below. Without convexity assumptions on the domains, we prove that the optimal map is smooth outside a relatively closed set of Hausdorff dimension at most $n-1$.
\end{abstract}
\maketitle

\section{Introduction}\label{introduction}

Let $\Omega,\Omega^*\subset\R^n$, $n\ge2$, be bounded open sets. Suppose  $f, g$ are  probability densities concentrated on $\Omega, \Omega^*$ respectively.
Let $\lambda,\Lambda$ be their lower and upper bounds
\begin{equation}\label{densities}
0<\lambda\le f,g\le\Lambda<\infty,\ \text{in their respective domains.}
\end{equation}

Optimal transport problem asks for finding the cheapest way to transport mass from $\Omega$ to $\Omega^*,$ namely solving the following minimization
problem:
$$\inf_{T_\sharp(fdx)=gdy}\int_{\Omega}\frac{1}{2}|x-Tx|^2f(x)dx,$$
where $T_\sharp fdx=gdy$ means that $\int_{T^{-1}A}f=\int_A g$ for any Borel set $A\subset \Omega^*.$
By Brenier's theorem \cite{B87, B91}, there exists a unique optimal map $T$. 
Moreover, $T=Du$ for some convex function $u.$ Hence,
\begin{equation}\label{transport}
(Du)_\#(fdx)=gdy.
\end{equation}
We always use the globally Lipschitz representative of $u$ described in Section 2.

The regularity theory for convex targets was developed by Caffarelli~\cite{C90,C91,C92}. For nonconvex targets the optimal map can be discontinuous even for smooth positive densities. Yu~\cite{Y07} studied the nondifferentiability set of planar transport potentials with a convex source and a nonconvex target. Under mild additional conditions, away from a countable set it is a disjoint union of countably many $C^1$ curves. Figalli~\cite{F10} proved partial regularity for general nonconvex planar domains. Figalli and Kim~\cite{FK10} established partial regularity in every dimension without convexity of the supports: $\{x\in \Omega: \partial u(x)\cap\Omega^*\ne\varnothing\}$ is an open set with full measure, and $u$ is smooth and strictly convex there when the densities are smooth. De Philippis and Figalli~\cite{DPF15} extended this to general costs. A very natural open problem was raised in \cite{DPF14}: Can we obtain a finer estimate for the size of the singular set?

The singularities of convex functions have been extensively studied. We recall two results particularly relevant to the present work.
Zaj\'i\v{c}ek~\cite{Z79} proved that the nondifferentiability set of a convex function has Hausdorff dimension at most $n-1$. Mooney~\cite{Mo15} proved that an Alexandrov solution of $\det D^2u\ge 1$ is strictly convex outside a set of zero $\HH^{n-1}$-measure. 

For optimal transport problem, under additional geometric assumptions, finer descriptions are known: see Kitagawa and McCann~\cite{KM19} for separated target components, Chen and Liu~\cite{CL25} for two disjoint convex targets, and Chen, Li and Liu~\cite{CLL26} for nonconvex polygonal targets in the plane, and Chen et al.~\cite{CLLLW26} for nonconvex polyhedral targets in arbitrary dimensions.
In the current paper, $\partial\Omega^*$ is an arbitrary smooth hypersurface.

 Let
\[
\Omega'=\bigl\{x\in\Omega:\exists\,r>0,\ B_r(x)\subset\Omega, \ u\in C^\infty(B_r(x)),\ D^2u>0\text{ on }B_r(x)\bigr\}
\]
be the \emph{regular set}. The \emph{singular set} is $\Sigma=\Omega\setminus\Omega'$. The main result of this paper is the following Hausdorff dimension estimate of
the singular set.

\begin{theorem}\label{main}
Let $\Omega,\Omega^*\subset\R^n$ be bounded connected open sets with $C^\infty$ boundary, $n\ge2$. Suppose $f\in C^\infty(\Omega)$ and $g\in C^\infty(\Omega^*)$ 
are probability densities satisfying \eqref{densities}. Then the potential $u$ in \eqref{transport} satisfies
\[
\dim_{\mathrm H}\Sigma\le n-1.
\]
\end{theorem}

The exponent $n-1$ is sharp, as a symmetric version of Caffarelli's smooth dumbbell construction~\cite{C92} yields an $(n-1)$-dimensional set of non-differentiability points for connected smooth domains and constant densities.

Let
\begin{equation}\label{singularsets}
A=\{x\in\Omega:u\text{ is not differentiable at }x\} \quad\text{and}\quad E=\{x\in\Omega\setminus A:Du(x)\in\partial\Omega^*\}.
\end{equation}
By the regularity theory of Figalli and Kim (Lemma~\ref{regularset}), $\Sigma=A\cup E$, and $\dim_{\mathrm H}A\le n-1$ by Zaj\'i\v{c}ek~\cite[Theorem~1]{Z79}. Hence Theorem~\ref{main} is equivalent to the following statement, which is the key of the paper.

\begin{proposition}\label{boundarydimension}
Under the hypotheses of Theorem~\ref{main},
\[
\dim_{\mathrm H}E\le n-1.
\]
\end{proposition}

We outline the proof of Proposition~\ref{boundarydimension}. For $0<\alpha<1$, a density bound for $\mu_{u+|\cdot|^2/2}$ yields
\[
c|x'|^{2/\alpha}\le u(x+x')-u(x)-Du(x)\cdot x'\le C|x'|^{1+\alpha}
\]
outside a set of zero $\HH^{n-1+\alpha}$-measure. 
These estimates allow us to localize a fixed dual potential \(v\) and obtain uniform control of its sections centered at \(Du(x)\), for \(x\) in a suitable countable family of subsets of \(E\). If a section of $v$ of height $h$ has semiaxes $a_1\ge\cdots\ge a_n$, the associated source ellipsoid has semiaxes $h/a_i$ and covering cost
\[
C_{\mathrm{cov}}=\frac{h^{n-1}}{a_2\cdots a_n}.
\]
Comparing this cost with finite weighted measures separates the positive upper mass limit from the zero mass limit. A selection argument treats the former. In the latter, two volume estimates in dimension two, and an additional measure on the subgradient graph in higher dimensions, yield a rank-one limit. The limiting boundary polynomial depends on one variable, or one component of the limiting source gradient takes finitely many values. On the limiting source line, all limiting positions arising from
points of $E$, except for finitely many interval endpoints, lie
in the interiors of intervals on which
$\psi(t)=\tilde{u}(te)$ is affine. Fenchel's inequality then implies
\[
r_{\tau h}(x')\ge c\theta r_h(x)
\]
outside finitely many slabs of relative width $\theta$. Covering the slabs by small balls and replacing $h$ by $\tau h$ on the complement contracts the total weight $r^{n-1+\delta/2}h^{\delta/(2N)}$. The corresponding supremum of normalized Hausdorff content satisfies $M\le M/2<\infty$, hence vanishes.

Section~\ref{preliminaries} records the preliminary facts, and Section~\ref{localization} establishes the local growth and section estimates. Section~\ref{measures} proves the selection estimates and fixes the subsets used in the measure comparisons. Section~\ref{planar} treats the planar limit and the covering argument used in every dimension. Section~\ref{higherdimensions} introduces the graph measure and completes the higher-dimensional proof.

\section{Preliminaries}\label{preliminaries}

For a proper lower semicontinuous convex function $w$ and $x\in\operatorname{dom}w$, let
\[
\partial w(x)=\{p\in\R^n:w(x')\ge w(x)+p\cdot(x'-x)\ \forall x'\} \quad\text{and}\quad w^*(p)=\sup_x\{p\cdot x-w(x)\},
\]
and $\epi w=\{(x,t):t\ge w(x)\}$. Then $w^{**}=w$ and
\begin{equation}\label{duality}
p\in\partial w(x) \iff x\in\partial w^*(p) \iff w(x)+w^*(p)=x\cdot p.
\end{equation}

For a finite convex function $w$ on an open convex set, its Monge-Amp\`ere measure
\[
\mu_w(\mathcal A)=|\partial w(\mathcal A)|
\]
is a locally finite Borel measure, weakly continuous under local uniform convergence~\cite[Sections~2.1--2.2]{F17}. By Alexandrov's theorem, at almost every $x$ there is a symmetric matrix $D^2w(x)\ge0$ with
\[
w(x+x')=w(x)+Dw(x)\cdot x'+\tfrac12x'^{\mathsf{T}}D^2w(x)x'+o(|x'|^2),
\]
and the absolutely continuous part of $\mu_w$ has density $\det D^2w$. The distributional Hessian of $w$ is a nonnegative matrix-valued Radon measure. Its trace $\Delta w$ has absolutely continuous part $\tr D^2w$.

Let $u$ be a convex solution to \eqref{transport}.
We extend $u$ to $\R^n$ as follows
$$\tilde u(x):=\sup\{L(x): L\text{ is affine},\ L\le u\text{ in }\Omega,\ DL\in\Omega^*\}.$$
 The extension $\tilde u$ is finite and Lipschitz and agrees with $u$ on $\Omega$. 
We henceforth denote this extension by $u$. 
By mass balance, we have that $Du\in\overline{\Omega^*}$ at every differentiability point in $\Omega$.

The following decomposition follows from Figalli and Kim~\cite[Proposition~2, the subsequent discussion, and Theorem~3.3]{FK10} and Zaj\'i\v{c}ek~\cite[Theorem~1]{Z79}.

\begin{lemma}\label{regularset}
Under the hypotheses of Theorem~\ref{main},
\begin{equation}\label{decomposition}
\Omega'=\{x\in\Omega\setminus A:Du(x)\in\Omega^*\} \quad\text{and}\quad \Sigma=A\cup E.
\end{equation}
The set $A$ lies in countably many $(n-1)$-dimensional Lipschitz graphs, and $\dim_{\mathrm H}A\le n-1$.
\end{lemma}

Let $w$ be a proper lower semicontinuous convex function, suppose $x_0\in \partial w(y_0)$, and $h>0$. The \emph{section} of $w$ at $(y_0,x_0)$ of height $h$ is
\begin{equation}\label{sections}
S_h[w](y_0,x_0)=\{y:w(y)-w(y_0)-x_0\cdot(y-y_0)<h\}.
\end{equation}
At a differentiability point, write $S_h[w](y_0)=S_h[w](y_0,Dw(y_0))$. The section is \emph{proper} if its closure is compactly contained in the open set on which $w$ is considered.

By John's theorem~\cite[Theorem~III]{J48}, a bounded convex set $\mathcal S$ with nonempty interior has a \emph{normalizing ellipsoid} $z+TB_1$ with
\begin{equation}\label{normalization}
z+TB_1\subset\mathcal S\subset z+nTB_1 \quad\text{and}\quad \det T>0.
\end{equation}
We call $T$ a \emph{normalizing matrix} of $\mathcal S$. Up to a rotation of coordinates, $T$ is $\diag(a_1,\ldots,a_n)$, with $a_1\ge\cdots\ge a_n>0$.  Set $\beta=a_1=\|T\|$, where $\|T\|$ is the operator norm.

The following two estimates will be used later. If $w$ is convex on a bounded convex domain and $\mu_w\ge c_0\,dx$, then every section $S_h$, not necessarily proper, satisfies \cite[Lemma~2.2]{Mo15}
\begin{equation}\label{sectionvolume}
|S_h|\le C(n,c_0)h^{n/2}.
\end{equation}
For any proper section of a finite convex function, the Alexandrov estimate~\cite[Lemma~2.5]{Mo15} implies
\begin{equation}\label{sectionmass}
\mu_w(S_h)\ge c_nh^n/|S_h|.
\end{equation}

\section{Localization and section estimates}\label{localization}

We first remove one exceptional set and localize the dual potential. On the remaining points, uniform power bounds for $u$ yield uniform estimates for its dual sections. We use only the density bounds in \eqref{densities} until Subsection~\ref{taylor}.

The argument in Caffarelli~\cite[Lemma~2]{C92}, as recalled in \cite[Section~3]{FK10}, yields
\begin{equation}\label{sourcemeasure}
\mu_u\ge(\lambda/\Lambda)\mathbf1_\Omega\,dx\qquad\text{in }\R^n.
\end{equation}

\subsection{Pointwise growth}
{~}

Choose a ball $B_R$ containing $\overline\Omega$ and consider
\[
\widetilde\mu=\mu_{u+|\cdot|^2/2}.
\]
This measure is finite on $B_R$ and satisfies $\widetilde\mu\ge\Delta u$. Indeed, for smooth convex functions,
\[
\det(I+D^2u)\ge\tr D^2u.
\]
Mollification and weak convergence of Monge-Amp\`ere measures~\cite[Section~2.2]{F17} extend this inequality to convex $u$. For $0<\alpha<1$, define
\[
Z_\alpha=\left\{x\in\Omega: \limsup_{r\downarrow0}\frac{\widetilde\mu(B_r(x))}{r^{n-1+\alpha}}=\infty\right\}.
\]
For a differentiability point $x$, write
\[
w_x(x')=u(x+x')-u(x)-Du(x)\cdot x'.
\]
The measure $\widetilde\mu$ controls the lower growth through section volumes and the upper growth through $\Delta u$. For the lower bound we use the section argument of Mooney~\cite[Section~3]{Mo15}, keeping track of the distance to one point of the section.

\begin{lemma}\label{growth}
We have $\HH^{n-1+\alpha}(Z_\alpha)=0$. Every $x\in\Omega\setminus(A\cup Z_\alpha)$ admits $c,C,r_0>0$, depending on $x,n,\alpha,\lambda,\Lambda$ such that, for $0<|x'|,r<r_0$,
\begin{equation}\label{growthbounds}
c|x'|^N\le w_x(x')\le C|x'|^{1+\alpha} \quad\text{and}\quad \partial u(B_r(x))\subset B_{Cr^\alpha}(Du(x)), \quad\text{where} \ N=2/\alpha.
\end{equation}
The supporting plane to $u$ at $x$ touches its graph only at $x$ in $\Omega$.
\end{lemma}

\begin{proof}
The density comparison theorem~\cite[Chapter~1, Theorem~3.3]{S14} gives $\HH^{n-1+\alpha}(Z_\alpha)=0$.

At a point $x\notin Z_\alpha$ there are $M_x\ge1$ and $0<r_x<\operatorname{dist}(x,\partial\Omega)$ such that
\begin{equation}\label{densitybound}
\widetilde\mu(B_r(x))\le M_xr^{n-1+\alpha}\qquad(0<r<r_x).
\end{equation}
Fix such a differentiability point and write $w=w_x$, so $w\ge0$ and $w(0)=0$.

For $0<t<r_x$, by a standard smooth approximation and divergence theorem we have
\[
m(t):=\frac1{|\partial B_t|}\int_{\partial B_t}w\,d\HH^{n-1}
=\frac1{|\mathbb S^{n-1}|}\int_0^t
\frac{\Delta w(B_s)}{s^{n-1}}\,ds
\le CM_xt^{1+\alpha}.
\]
Since $w\ge0$, by the mean value inequality we have
\[
\sup_{B_{2r}}w\le\frac1{|B_r|}\int_{B_{3r}}w\,dz
\le CM_xr^{1+\alpha}\qquad(3r<r_x).
\]
This proves the upper growth bound. If $x'\in B_r$ and
$0\ne p\in\partial w(x')$, by the supporting inequality at
$x'+rp/|p|\in B_{2r}$ we have
\[
r|p|\le w(x'+rp/|p|)-w(x')\le\sup_{B_{2r}}w
\le CM_xr^{1+\alpha},
\]
from which the subgradient inclusion follows.

Fix a small $x'\ne0$,  choose a sufficiently small height $h>w(x')$. Consider
\[
S=\{z:w(z)+\tfrac12|z|^2<h\} \quad\text{and}\quad T=\{z\in B_{r_x}:w(z)<h\}.
\]
For small $h$, $x+S$ is a proper section of $u+|\cdot|^2/2$ and $S\subset B_{\sqrt{2h}}$. The set $T$ is a section of $w$ on the fixed ball $B_{r_x}$, contains $S$ and $x'$, and satisfies by \eqref{sourcemeasure} and \eqref{sectionvolume},
\begin{equation}\label{volumebound}
|T|\le Ch^{n/2} .
\end{equation}
Let $Q$ be the projection of $S$ onto $(x')^\perp$. Convexity gives
$\tfrac12S+[0,x'/2]\subset T.$ It follows that 
$$\HH^{1}(T\cap(X+\R x'))\geq \frac{|x'|}{2},\qquad\forall X\in\tfrac12Q.$$
 On the other hand, $S\subset B_{\sqrt{2h}}$ bounds every line section of $S$ by $2\sqrt{2h}$ in length. By Fubini's theorem,
\[
|T|\ge2^{-n}|x'|\HH^{n-1}(Q) \quad\text{and}\quad |S|\le2\sqrt{2h}\,\HH^{n-1}(Q).
\]
Combining the above estimates with \eqref{volumebound} we have
\[
|S|\le C\frac{\sqrt h}{|x'|}|T| \le C\frac{h^{(n+1)/2}}{|x'|}.
\]
It then follows by \eqref{sectionmass} and \eqref{densitybound} that
\[
c|x'|h^{(n-1)/2} \le\frac{c_nh^n}{|S|} \le\widetilde\mu(x+S) \le\widetilde\mu(B_{\sqrt{2h}}(x)) \le C M_xh^{(n-1+\alpha)/2}.
\]
Consequently $h\ge(CM_x)^{-2/\alpha}|x'|^{2/\alpha}$. Since this holds for every sufficiently small $h>w(x')$, it follows that
\[
w(x')\ge(CM_x)^{-2/\alpha}|x'|^{2/\alpha}.
\]This proves the lower bound and strict convexity at $x.$
\end{proof}

Splitting the remaining points according to rational bounds for $c,C,r_0$, we obtain countably many subsets on which \eqref{growthbounds} holds uniformly.

\subsection{A fixed dual potential}
{~}

Using the ball $B_R\supset\overline\Omega$ fixed above, define
\begin{equation}\label{fixeddual}
v(y)=\max_{x\in\overline{B_R}}\{x\cdot y-u(x)\}.
\end{equation}
This function is finite, Lipschitz and convex. By biconjugation, $v^*=u$ on $B_R$~\cite[Appendix~A.4]{F17}. We use this same function throughout the proof. The next lemma localizes its equation.

\begin{lemma}\label{duallocalization}
Let $x_0\in\Omega\setminus(A\cup Z_\alpha)$ and $y_0=Du(x_0)$. There is a bounded neighborhood $\Omega_0^*$ of $y_0$ such that $v=u^*$ on $\Omega_0^*$ and $\partial v(\overline{\Omega_0^*})\Subset\Omega$. Moreover $Dv(y_0)=x_0$, and for every Borel set $\mathcal A\subset\Omega_0^*$,
\begin{equation}\label{dualmass}
\int_{\partial v(\mathcal A)}f(x)\,dx =\int_{\mathcal A\cap\Omega^*}g(y)\,dy.
\end{equation}
In particular,
\begin{equation}\label{dualequation}
\frac\lambda\Lambda\mathbf1_{\Omega^*}\,dy\le \mu_v \le\frac\Lambda\lambda\mathbf1_{\Omega^*}\,dy \qquad\text{on }\Omega_0^*.
\end{equation}
\end{lemma}

\begin{proof}
Choose $r>0$ with $\overline{B_{2r}(x_0)}\Subset\Omega$ such that $w_{x_0}(z)\ge c|z|^N$ for $|z|\le2r$. Convexity along rays extends this to
\[
w_{x_0}(z)\ge cr^{N-1}|z|\qquad(|z|\ge r).
\]
Hence, if $|y-y_0|\le cr^{N-1}/2$, then
\[
(y-y_0)\cdot z-w_{x_0}(z)\le-\tfrac c2r^{N-1}|z|<0\qquad(|z|\ge r),
\]
whereas the left side is zero at $z=0$. 
Thus, taking
$\Omega_0^*=B_{cr^{N-1}/2}(y_0),$
all maximizers lie in $B_r(x_0)$ for every
$y\in\overline{\Omega_0^*}$. Consequently,
$v=u^*$ on $\Omega_0^*$ and
$\partial v(\overline{\Omega_0^*})
\subset B_r(x_0)\Subset\Omega.$
At $y_0$ the unique maximizer is $x_0$, so
$Dv(y_0)=x_0$.

Since $v=u^*$ on $\Omega_0^*$, by Legendre duality we have
\[
x\in\partial v(\mathcal A)
\quad\Longleftrightarrow\quad
Du(x)\in\mathcal A
\] for almost every $x\in\Omega.$
Hence by \eqref{transport} and \eqref{densities},
we have \eqref{dualmass} and \eqref{dualequation}.
\end{proof}

Shrink the neighborhoods in Lemma~\ref{duallocalization} to bounded boundary charts $\Omega_0^*$ crossing $\partial\Omega^*$. Thus $v=u^*$ on $\Omega_0^*$, $\partial v(\overline{\Omega_0^*})\Subset\Omega$, and \eqref{dualequation} holds there. In each chart choose orthonormal coordinates $y=(y',y_n)$ and a smooth function $\varrho$ such that
\[
\Omega^*\cap\Omega_0^*=\{y\in\Omega_0^*:y_n>\varrho(y')\}.
\]
Define $F(y)=y_n-\varrho(y')$. Then
\[
\Omega^*\cap\Omega_0^*=\{F>0\}\quad\text{and}\quad \partial\Omega^*\cap\Omega_0^*=\{F=0\}.
\]
Since $DF(y)=(-D\varrho(y'),1)$, we have
\begin{equation}\label{boundarygradient}
|DF(y)|=\sqrt{1+|D\varrho(y')|^2}\ge1.
\end{equation}
By second countability, countably many such charts cover $Du(E\setminus Z_\alpha)$. On slightly smaller charts the derivatives of $\varrho$ of each fixed order are bounded.

\begin{definition}\label{preparedset}
Fix $0<\alpha<1$, take $N=2/\alpha$, and choose one such chart $\Omega_0^*$. For fixed $c,C,r_0>0$, let $\mathcal P=\mathcal P(c,C,r_0,\Omega_0^*)$ consist of the points $x\in E\setminus Z_\alpha$ satisfying \eqref{growthbounds} and
\[
\overline{B_{r_0}(x)}\subset\Omega \quad\text{and}\quad \overline{B_{r_0}(Du(x))}\subset\Omega_0^*.
\]
\end{definition}

By Lemmas~\ref{growth} and~\ref{duallocalization}, every point of $E\setminus Z_\alpha$ belongs to a set $\mathcal P(c,C,r_0,\Omega_0^*)$ from Definition~\ref{preparedset}. Choosing the charts and bounds from countable families therefore gives a countable cover of $E\setminus Z_\alpha$ by such sets. We work on one such set at a time. All subsequent local estimates are uniform in its base point and height.

\subsection{Size of dual sections}\label{sectionsize}
{~}

Fix a set $\mathcal P$ as in Definition~\ref{preparedset} and use $v$ from \eqref{fixeddual}. Let $T$ normalize $S_h[v](y_0,x_0)$. Normalizing the target coordinates by $T^{-1}$ and the potential by $h^{-1}$ induces the source coordinate change $X=h^{-1}T^{\mathsf{T}}(x-x_0)$. We call $x_0+hT^{-\mathsf{T}}B_1$ the \emph{source ellipsoid} associated with this normalization. Its semiaxes are $h/a_i$. For $x\in\mathcal P$ set
\begin{equation}\label{radii}
\mathcal R_h(x)=\sup\bigl\{|y-Du(x)|:y\in S_h[v](Du(x),x)\bigr\} \quad\text{and}\quad r_h(x)=c_*\frac h{\mathcal R_h(x)}.
\end{equation}
Here $c_*\in(0,1]$ depends only on the fixed data and will be chosen in Subsection~\ref{comparison}. The estimates below hold for any fixed such choice. The radius $r_h(x)$ is comparable to the smallest semiaxis of $x+hT^{-\mathsf{T}}B_1$.

Fix a base point $x_0\in\mathcal P$ and put $y_0=Du(x_0)$. Translating $y_0$ and $x_0$ to the origin and subtracting the supporting affine functions, we may assume $u(0)=v(0)=0$, $Du(0)=Dv(0)=0$ and $u,v\ge0$, so that $S_h=S_h[v](0,0)=\{v<h\}$. Let $y_h+T_hB_1$ be the normalizing ellipsoid of $S_h$, with smallest and largest semiaxes $a_{\min}(h)$ and $\beta(h)$, and write $\mathcal R_h=\mathcal R_h(0)$, $r_h=r_h(0)$.

\begin{lemma}\label{sectionbounds}
Uniformly for base points in $\mathcal P$ and for small $h$, the section $S_h$ is proper and contained in $\Omega_0^*$, and
\begin{gather}
c_1|y|^{(1+\alpha)/\alpha}\le v(y)\le C_1|y|^{N/(N-1)}, \label{dualgrowth}\\
ch^{1-1/N}\le a_{\min}(h)\le\beta(h)\le Ch^{\alpha/(1+\alpha)} ,\qquad \beta(h)\le\mathcal R_h\le2n\beta(h),\label{axisbounds}\\
|\det T_h|\ge ch^{n/2},\label{determinant}\\
ch^{1/(1+\alpha)}\le r_h\le Ch^{1/N} ,\qquad \|hT_h^{-\mathsf{T}}\|\le Ch^{1/N}.\label{radiusbounds}
\end{gather}
Moreover, $h\mapsto r_h$ is continuous and nondecreasing and attains every sufficiently small positive value.
\end{lemma}

\begin{proof}
By the definition of $v$ and the upper bound in \eqref{growthbounds}, we have
\[
v(y)\ge\sup_{0<t<r_0}\bigl(t|y|-Ct^{1+\alpha}\bigr)=c_1|y|^{(1+\alpha)/\alpha}
\]
for small $|y|$. 
For the upper bound, by the proof of Lemma \ref{duallocalization}, for $|y|<\frac12cr_0^{N-1}$ the supremum defining $v(y)$ is therefore attained in $B_{r_0}$, where it is bounded by $\sup_{t\ge0}(t|y|-ct^N)=C_1|y|^{N/(N-1)}$. This proves \eqref{dualgrowth}, from which \eqref{axisbounds} follows.

By \eqref{sectionmass} and \eqref{dualequation},
\[
ch^n/|S_h|\le\mu_v(S_h)\le C|S_h|,
\]
so $|S_h|\ge ch^{n/2}.$ Then,  \eqref{determinant} follows from $\det T_h\approx |S_h|.$
The estimates \eqref{radiusbounds} follow from $r_h=c_*h/\mathcal R_h$, \eqref{axisbounds} and $\|hT_h^{-\mathsf{T}}\|=h/a_{\min}(h)$.

Since $v(0)=0$, by convexity we have 
\[
\frac{h_1}{h_2}S_{h_2}\subset S_{h_1}\subset S_{h_2}\qquad(h_1<h_2),
\]
 Hence
\[
\frac{h_1}{h_2}\mathcal R_{h_2}\le\mathcal R_{h_1}\le\mathcal R_{h_2},
\]
so $\mathcal R_h$ is continuous and $r_h$ is continuous and nondecreasing. Since $r_h\to0$, it attains every sufficiently small positive value.
\end{proof}

\subsection{Rescaling the target boundary}\label{taylor}
{~}

Recall that $F$ defines the fixed boundary chart. Take $y_0=Du(x_0)$ with $x_0\in\mathcal P$. Choose an integer $d_0$ such that
\begin{equation}\label{degree}
d_0>(1+\alpha)/\alpha.
\end{equation}
Let $P_h^0$ be the degree-$d_0$ Taylor polynomial at zero of $Y\mapsto F(y_0+T_hY)$, and set
\[
\gamma_h=\|P_h^0\|_{L^\infty(B_1)}.
\]
Consider a sequence $(x_j,y_j,h_j)$ with $x_j\in\mathcal P$, $h_j\downarrow0$ and $y_j=Du(x_j)$. For each $j$, let $T_j$ normalize $S_{h_j}[v](y_j,x_j)$, $P_j^0$ be the degree-$d_0$ Taylor polynomial at zero of $Y\mapsto F(y_j+T_jY)$, and $\gamma_j=\|P_j^0\|_{L^\infty(B_1)}$. The following normalization keeps a nonzero polynomial in the limit.

\begin{lemma}\label{taylorlimit}
For every $x_0\in\mathcal P$ and sufficiently small $h>0$, we have
\begin{equation}\label{amplitude}
ca_{\min}(h)\le\gamma_h\le C\beta(h).
\end{equation}
For every sequence as above, there is a subsequence such that, for every $L>0$,
\[
\left\|\frac{F(y_j+T_j\,\cdot\,)}{\gamma_j}-P\right\|_{C^2(B_L)}\to0,
\]
where $P$ is a polynomial with $\deg P\le d_0$, $P(0)=0$ and $\|P\|_{L^\infty(B_1)}=1$.
\end{lemma}

\begin{proof}
On the space of polynomials of degree at most $d_0$, the coefficient norm and the $L^\infty(B_1)$ norm are equivalent. Since the linear coefficient of $P_h^0$ is $T_h^{\mathsf{T}}DF(y_0)$, we find by \eqref{boundarygradient} that
\[
\gamma_h\ge c|T_h^{\mathsf{T}}DF(y_0)|\ge ca_{\min}(h).
\]
Since $F(y_0)=0$, the constant coefficient vanishes. For $1\le k\le d_0$, each order-$k$ coefficient is bounded by
\[
C\|D^kF(y_0)\|\,\|T_h\|^k\le C\beta(h)^k,
\]
since the derivatives of $F$ are uniformly bounded. Thus $\gamma_h\le C\beta(h)$ for small $h$. All constants are uniform for $x_0\in\mathcal P$ and in the orthonormal frame. Hence \eqref{amplitude} follows.

Fix a sequence $x_j\in\mathcal P$, $h_j\to0$. Let $\beta_j=\|T_j\|$. By \eqref{amplitude} and \eqref{axisbounds}, $\gamma_j\ge ch_j^{1-1/N}$. For each fixed $L>0$, this lower bound, Taylor's formula and \eqref{axisbounds} show
\[
\left\|\frac{F(y_j+T_j\,\cdot\,)-P_j^0}{\gamma_j}\right\|_{C^2(B_L)} \le C_L\frac{\beta_j^{d_0+1}}{\gamma_j} \le C_Lh_j^{\alpha(d_0+1)/(1+\alpha)-(1-1/N)}\to0.
\]
Here $C_L$ is independent of $j$, and the exponent is positive by \eqref{degree}. The polynomials $P_j^0/\gamma_j$ have degree at most $d_0$, vanish at zero, and have $L^\infty(B_1)$ norm one. Thus $P_j^0/\gamma_j\to P$ in $C^2(B_L)$ for every fixed $L>0$, with $P(0)=0$ and $\|P\|_{L^\infty(B_1)}=1$.
\end{proof}

\section{Finite measures and coverings}\label{measures}

Fix a set $\mathcal P=\mathcal P(c,C,r_0,\Omega_0^*)$ as in Definition~\ref{preparedset}. Choose
\[
\Omega_0=\{x\in\Omega: \operatorname{dist}(x,\partial\Omega)>r_0/2\}, \quad\text{and}\quad \Omega_1^*=\{y\in\Omega_0^*: \operatorname{dist}(y,\partial\Omega_0^*)>r_0/2\}.
\]
By Definition~\ref{preparedset},
\[
\mathcal P\subset\Omega_0\quad\text{and}\quad Du(\mathcal P)\subset\Omega_1^*,
\]
and these two regions will contain the source and target ellipsoids, respectively, at sufficiently small heights. For every $x_0\in\mathcal P$, let $T_h$ normalize $S_h[v](y_0,x_0)$ with $y_0=Du(x_0)$. By Definition~\ref{preparedset} and Lemma~\ref{sectionbounds}, there exists $h_0>0$, independent of $x_0$, such that for all
\[
0<h<h_0,
\]
there holds
\begin{equation}\label{containment}
S_h[v](y_0,x_0)\Subset\Omega_1^*,\quad x_0+hT_h^{-\mathsf{T}}B_1\Subset\Omega_0,\quad \partial u(x_0+hT_h^{-\mathsf{T}}B_1)\subset\Omega_1^*.
\end{equation}
The last inclusion follows from the uniform subgradient estimate in \eqref{growthbounds}, since the source ellipsoid shrinks uniformly and $\operatorname{dist}(Du(x_0),\partial\Omega_0^*)>r_0$.

We henceforth consider only heights $0<h<h_0$.

\subsection{Weighted measures}\label{weights}
{~}

Let $0<\epsilon<1$. Define the target measure $\nu_0$ on $\Omega_1^*$ and the source measure $\mu_0$ on $\Omega_0$ by
\begin{equation}\label{weightedmeasures}
d\nu_0(y)=\mathbf1_{\{F(y)\ne0\}}|F(y)|^{-1+\epsilon}\,dy \quad\text{and}\quad d\mu_0(x)=\mathbf1_{\{Du(x)\in\Omega_1^*\cap\Omega^*\}} F(Du(x))^{-1+\epsilon}\,dx.
\end{equation}
The boundary chart and the transport identity imply
\begin{equation}\label{finitemeasures}
\nu_0(\Omega_1^*)<\infty \quad\text{and}\quad \mu_0(\Omega_0)\le\frac\Lambda\lambda \int_{\Omega_1^*\cap\Omega^*}F(y)^{-1+\epsilon}\,dy<\infty.
\end{equation}
For a section of height $h$ with normalizing matrix $T$ and semiaxes $a_1\geq\dots\geq a_n$ for the associated John ellipsoid, define the \emph{covering cost} of this section by
\[
\cc:=\left(\frac{h}{a_1}\right)^{n-1}\prod_{i>1}\frac{h/a_i}{h/a_1}=\frac{h^{n-1}}{a_2\cdots a_n}.
\]
Here $(h/a_1)^{n-1}$ represents the cost per ball and $\prod_{i>1}(a_1/a_i)$ the number of balls needed to cover the source ellipsoid, both up to dimensional constants.

In the planar proof we use $\nu_0$ and $\mu_0$. In higher dimensions, we retain $\nu_0$ and use the graph measures introduced in Section~\ref{higherdimensions} in place of $\mu_0$.

\subsection{Selection}\label{selection}
{~}

For a section $S_h[v](y_0,x_0)$, let $z+TB_1$ be a normalizing ellipsoid. Write $y=y_0+TY$ and define
\begin{equation}\label{normalizedpotential}
\tilde{v}(Y)=\frac{v(y_0+TY)-v(y_0)-x_0\cdot TY}{h}.
\end{equation}
We call $\tilde{v}$ the \emph{normalized potential}. Then $\tilde{v}\ge0=\tilde{v}(0)$ and $T^{-1}(S_h[v](y_0,x_0)-y_0)=\{\tilde{v}<1\}$.
Set $Y^\circ:=T^{-1}(z-y_0)$. By \eqref{normalization} and $y_0\in S_h[v](y_0,x_0)$,
\[
B_1(Y^\circ)\subset\{\tilde{v}<1\}\subset B_{2n}\quad\text{and}\quad |Y^\circ|\le n.
\]

We use Macbeath regions~\cite{Ma52} and the packing--covering
scheme of B\'ar\'any and Larman~\cite{BL88}.
For epigraph packing in strips of fixed vertical width,
see~\cite{AFM26}.
We give the formulation needed here, including the selected-mass
estimate and simultaneous source and target coverings.

\begin{lemma}\label{selectionlemma}
Let $v$ be a proper lower semicontinuous convex function, $\Omega_1^*$ a bounded open set, and $\nu$ a finite Borel measure on $\Omega_1^*$. Let $I$ be an index set. For $i\in I$, take $x_i\in\partial v(y_i)$, $h_i>0$, 
an invertible linear transformation $T_i$, and a point $Y_i\in\mathbb R^n$ in the normalized target coordinates $y=y_i+T_iY$, 
subject to the hypotheses below. Define
\begin{equation}\label{normalizedsequence}
\tilde{v}_i(Y)=\frac{v(y_i+T_iY)-v(y_i)-x_i\cdot T_iY}{h_i}.
\end{equation}
Suppose that, with constants $c,C,M,r_1,R_1$ uniform in $i\in I$,
\begin{equation}\label{localbounds}
\sup_{B_{r_1}(Y_i)}\tilde{v}_i\le M\quad\text{and}\quad|Y_i|\le R_1,
\end{equation}
and, for every $i\in I$ and $Z\in\R^n$,
\begin{equation}\label{symmetricgrowth}
\tilde{v}_i(Y_i+Z)+\tilde{v}_i(Y_i-Z) \ge c|Z|-C.
\end{equation}
Then there is $C_*>0$, depending only on these constants and $n$, with the following property. For Borel sets $Q_i\subset\Omega_1^*\cap\bigl(y_i+T_iB_{r_1/20}(Y_i)\bigr)$ and $h>0$, take $I_h=\{i\in I:h/2<h_i\le h\}$. There is a countable $J\subset I_h$ such that

\begin{itemize}
\item[(i)] The selected sets satisfy
\begin{equation}\label{selectedmass}
\sum_{j\in J}\nu(Q_j)\le C_*\,\nu(\Omega_1^*).
\end{equation}
\item[(ii)] The source and target points are covered by
\begin{equation}\label{selectedcover}
\{x_i:i\in I_h\}\subset\bigcup_{j\in J} \bigl(x_j+C_*h_jT_j^{-\mathsf{T}}B_1\bigr) \quad\text{and}\quad \{y_i:i\in I_h\}\subset\bigcup_{j\in J} \bigl(y_j+C_*T_jB_1\bigr).
\end{equation}
\end{itemize}

\end{lemma}

\begin{proof}
\emph{Step 1. Symmetric regions.} Denote $\mathcal K=\epi v\subset\R^{n+1}$. For $L_0\ge M+2$, let
\[
y_i^\circ=y_i+T_iY_i,\qquad z_i=\bigl(y_i^\circ,v(y_i^\circ)+L_0h_i\bigr),\qquad \mathcal N_i=(\mathcal K-z_i)\cap(z_i-\mathcal K).
\]
Thus $z_i+\mathcal N_i$ is the largest set contained in $\mathcal K$ that is centrally symmetric about $z_i$. Let $\ell_i(y):=v(y_i)+x_i\cdot(y-y_i)$. The coordinates $Y=T_i^{-1}(y-y_i)$ and $s=(t-\ell_i(y))/h_i$ from \eqref{normalizedsequence} send $\mathcal K$ to $\epi \tilde{v}_i$. Let $\widehat{\mathcal N}_i$ be the image of $z_i+\mathcal N_i$ in these coordinates. By the first estimate in \eqref{localbounds} and the convexity, we find that
\begin{equation}\label{innerbox}
B_{r_1}(Y_i)\times[H_i-1,H_i+1] \subset\widehat{\mathcal N}_i \quad \text{where} \ H_i=\tilde{v}_i(Y_i)+L_0.
\end{equation}
Take $(Y_i+Z,H_i+\sigma)\in\widehat{\mathcal N}_i$, for some $(Z,\sigma)\in\R^n\times\R$. By central symmetry, the reflected point $(Y_i-Z,H_i-\sigma)$ also belongs to $\widehat{\mathcal N}_i$. Both points lie in $\epi \tilde{v}_i$, hence
\[
\tilde{v}_i(Y_i+Z)\le H_i+\sigma \quad\text{and}\quad \tilde{v}_i(Y_i-Z)\le H_i-\sigma.
\]
Since $\tilde{v}_i\ge0$, these inequalities imply $ 0\le H_i+\sigma\le2H_i\le2M+2L_0$, which controls the vertical coordinate of $\widehat{\mathcal N}_i$. Adding the above two inequalities and using \eqref{symmetricgrowth} also lead to
\[
c|Z|-C\le \tilde{v}_i(Y_i+Z)+\tilde{v}_i(Y_i-Z) \le2H_i \le 2(M+L_0).
\]
Thus $|Z|$ is uniformly bounded. Enlarging $C$ if needed, these bounds give
\begin{equation}\label{outerbox}
\widehat{\mathcal N}_i \subset B_C(Y_i)\times [0,2M+2L_0].
\end{equation}
Taking horizontal projections in \eqref{innerbox} and \eqref{outerbox} and returning to the original coordinates,
\begin{equation}\label{projection}
y_i^\circ+T_iB_{r_1} \subset \{y:\text{there exists }t\text{ with }(y,t)\in z_i+\mathcal N_i\} \subset y_i^\circ+T_iB_C.
\end{equation}
The vertical interval bounds in \eqref{innerbox} and \eqref{outerbox}, rescaled by $h_i$, also give
\begin{equation}\label{verticallength}
2h_i \le\LL^1\{t:(y,t)\in z_i+\mathcal N_i\} \le Ch_i \qquad\bigl(\ \forall y\in y_i^\circ+T_iB_{r_1} \ \bigr).
\end{equation}
Here $C$ depends only on the constants in the hypotheses and the fixed choice of $L_0$.

\emph{Step 2. Selection.}
In the $(Y,s)$ coordinates, shrinking about $z_i$ acts as 
$(Z,\sigma)\to(Z/10,\sigma/10)$. 
Since $Q_i\subset y_i^\circ+T_iB_{r_1/20}$, multiplying its normalized horizontal displacements by $10$ places them in $B_{r_1/2}$. Thus \eqref{innerbox} gives
\begin{equation}\label{reducedlength}
\LL^1\{t:(y,t)\in z_i+\tfrac1{10}\operatorname{int}\mathcal N_i\} \ge h_i/5\qquad(y\in Q_i).
\end{equation}
For every $(y,t)\in z_i+\tfrac1{10}\operatorname{int}\mathcal N_i$, 
there is a point $(\bar y,\bar t)\in z_i+\mathcal N_i\subset\epi v$ such that
\[
 (y,t)=\tfrac9{10}z_i+\tfrac1{10}(\bar y,\bar t).
\]
Using the definition of $z_i$ and convexity, we obtain
\[
 t-v(y)
 \ge\bigl(\frac9{10}v(y_i^\circ)+\frac9{10}L_0h_i+\frac1{10}\bar t\bigr)
-\bigl(\frac9{10}v(y_i^\circ)+\frac1{10}v(\bar y)\bigr)
 \ge\frac9{10}L_0h_i.
\]
Since both $(y,t)$ and its reflection $\left(2y_i^\circ-y,2\bigl(v(y_i^\circ)+L_0h_i\bigr)-t\right)$ about $z_i$ lie above $\ell_i$, we have
\begin{equation}\label{heightbound}
t-v(y)\le t-\ell_i(y) \le2\bigl(v(y_i^\circ)+L_0h_i-\ell_i(y_i^\circ)\bigr) \le2(M+L_0)h_i.
\end{equation}
Thus, for $h/2<h_i\le h$, every $(y,t)\in z_i+\tfrac1{10}\operatorname{int}\mathcal N_i$ lies in the same strip
\begin{equation}\label{heightstrip}
ch\le t-v(y)\le Ch.
\end{equation}

Choose a maximal pairwise disjoint subfamily of $z_i+\tfrac1{10}\operatorname{int}\mathcal N_i$ with $h/2<h_i\le h$, indexed by $j$. It is countable. For each fixed $y$, their vertical intervals are disjoint and lie in the strip~\eqref{heightstrip}. Integrating their lengths over $Q_j\subset\Omega_1^*$ and using \eqref{reducedlength}, 
Fubini's theorem and \eqref{heightbound} in order, we have
\[
\begin{aligned}
\frac h{10}\sum_j\nu(Q_j) &\le\sum_j\int_{Q_j} \LL^1\{t:(y,t)\in z_j+\tfrac1{10}\operatorname{int}\mathcal N_j\} \,d\nu(y)\\
&\le\int_{\Omega_1^*}\LL^1\{t:ch\le t-v(y)\le Ch\}\,d\nu(y)\le Ch\,\nu(\Omega_1^*).
\end{aligned}
\]
Dividing by $h$ proves \eqref{selectedmass}.

\emph{Step 3. Target covering.}
Fix $i\in I_h$. By maximality in Step 2, there is $j\in J$ such that
\begin{equation}\label{intersection}
\bigl(z_i+\tfrac1{10}\operatorname{int}\mathcal N_i\bigr)
\cap\bigl(z_j+\tfrac1{10}\operatorname{int}\mathcal N_j\bigr)\ne\varnothing.
\end{equation}
The corresponding $1/5$-regions also intersect. 
The Macbeath overlap property~\cite[Lemma~2.5]{BCP93} gives
\begin{equation}\label{commoncenter}
z_i+\tfrac15\mathcal N_i\subset z_j+\mathcal N_j,
\qquad z_j+\tfrac15\mathcal N_j\subset z_i+\mathcal N_i.
\end{equation}
To apply the convex-body version, intersect $\mathcal K$ with a sufficiently large closed ball containing both full regions, which are bounded by \eqref{outerbox}. This leaves their Macbeath regions unchanged.
Taking difference sets in \eqref{commoncenter} and using central symmetry yields
\begin{equation}\label{regioncomparison}
\tfrac15\mathcal N_j\subset\mathcal N_i\subset5\mathcal N_j.
\end{equation}
Taking horizontal projections and using \eqref{projection}, we obtain
\[
cT_jB_1\subset T_iB_1\subset CT_jB_1.
\]
Also, \eqref{commoncenter} gives $z_i\in z_j+\mathcal N_j$, hence
$y_i^\circ-y_j^\circ\in CT_jB_1$ by \eqref{projection}.
Since $y_i^\circ-y_i=T_iY_i$ and $|Y_i|,|Y_j|\le R_1$,
\[
y_i-y_j=(y_i^\circ-y_j^\circ)-T_iY_i+T_jY_j\in CT_jB_1.
\]
This proves the target covering in \eqref{selectedcover}.

\emph{Step 4. Source covering.}
By \eqref{commoncenter},
\begin{equation}\label{l4.1:e1}
z_j+\tfrac15\mathcal N_j
\subset(z_i+\mathcal N_i)\cap(z_j+\mathcal N_j).
\end{equation}
By \eqref{projection}, its horizontal projection contains
\begin{equation}\label{l4.1:e0}
y_j^\circ+\tfrac{r_1}{5}T_jB_1.
\end{equation}
On $z_i+\mathcal N_i$, reflection about $z_i$, as in \eqref{heightbound}, gives
\[
0\le v(y)-\ell_i(y)\le t-\ell_i(y)\le2(M+L_0)h_i.
\]
The same bound holds with $j$ in place of $i$. Since $h_i,h_j\le h$, \eqref{l4.1:e1} therefore implies
\[
|\ell_i(y)-\ell_j(y)|\le Ch
\qquad\bigl(y\in y_j^\circ+\tfrac{r_1}{5}T_jB_1\bigr).
\]
The affine function $\ell_i-\ell_j$ has slope $x_i-x_j$, so its oscillation on \eqref{l4.1:e0} is
$\frac{2r_1}{5}|T_j^{\mathsf T}(x_i-x_j)|$.
As $h<2h_j$, we obtain
\[
|T_j^{\mathsf T}(x_i-x_j)|\le Ch_j.
\]
Thus $x_i\in x_j+Ch_jT_j^{-\mathsf T}B_1$, proving the source covering in \eqref{selectedcover} and completing (ii).
\end{proof}

To control source-side measures, we also apply Lemma~\ref{selectionlemma} to $v^*$. For ordinary sections, write $\tilde{u}_i=\tilde{v}_i^*$ and $x=x_i+h_iT_i^{-\mathsf{T}}X$. Then
\begin{equation}\label{sourcenormalization}
\begin{gathered}
\tilde{u}_i(X)=\frac{v^*(x_i+h_iT_i^{-\mathsf{T}}X)-v^*(x_i)-y_i\cdot h_iT_i^{-\mathsf{T}}X}{h_i}\ge0,\\
X\in\partial\tilde{v}_i(Y)\iff Y\in\partial\tilde{u}_i(X).
\end{gathered}
\end{equation}
Let $b_i+T_iB_1$ be the chosen normalizing ellipsoid and set $Y_i:=T_i^{-1}(b_i-y_i)$. By \eqref{normalization},
\[
B_1(Y_i)\subset\{\tilde{v}_i<1\}\subset B_{2n}\quad\text{and}\quad |Y_i|\le n.
\]
Convexity along $\{\lambda Y:\lambda\ge0\}$ gives
\begin{equation}\label{raybound}
\tilde{v}_i(Y)\ge |Y|/(2n)-1.
\end{equation}
Since $\tilde{u}_i=\tilde{v}_i^*$, we thus see from the definition of Legendre transform that
\begin{equation}\label{s4:e1}
0\le\tilde{u}_i(X)\le\sup_{Y\in\mathbb R^n}\left\{X\cdot Y-\frac{|Y|}{2n}+1\right\}=1
\qquad\text{for}\ X\in B_{1/(2n)}.
\end{equation}
Since $B_1(Y_i)\subset\{\widetilde v_i<1\}$, we have
\[
\tilde{u}_i(X)\ge \sup_{Y\in B_1(Y_i)} X\cdot Y-1 =  Y_i\cdot X+|X|-1.
\]
Switching $X$ and $-X$ then shows
\begin{equation}\label{sourcegrowth}
\tilde{u}_i(X)+\tilde{u}_i(-X)\ge2|X|-2.
\end{equation}
We now fix the constant $\rho$ as below which is used throughout the paper,
\begin{equation}\label{rhochoice}
0<\rho<\frac1{40n}.
\end{equation}
Together with \eqref{s4:e1},
\eqref{sourcegrowth} verifies the hypotheses of Lemma~\ref{selectionlemma} for $v^*$ with base points $x_i$, supporting slopes $y_i$, normalizing matrices $h_iT_i^{-\mathsf{T}}$ and normalized centers zero. The local boundedness radius is $1/(2n)$, so \eqref{rhochoice} verifies the prescribed set condition. For a finite source measure $\mu$ and prescribed Borel subsets of $x_i+\rho h_iT_i^{-\mathsf{T}}B_1$, the lemma selects a countable subfamily with bounded total $\mu$-mass and covers the source centers by $x_j+Ch_jT_j^{-\mathsf{T}}B_1$.

For $s>0$, let $c_s$ be the standard normalization constant for $\HH^s$. For $\mathcal A\subset\R^n$ and $0<\eta\le\infty$, define
\[
\HH^s_\eta(\mathcal A)=c_s\inf\left\{\sum_{i=1}^\infty(\operatorname{diam}U_i)^s:\mathcal A\subset\bigcup_{i=1}^\infty U_i,\ \operatorname{diam}U_i\le\eta\right\}.
\]
The infimum is over countable covers by sets $U_i\subset\R^n$. The case $\eta=\infty$ imposes no diameter restriction and gives the $s$-dimensional Hausdorff content $\HH^s_\infty$. The Hausdorff measure satisfies $\HH^s(\mathcal A)=\lim_{\eta\downarrow0}\HH^s_\eta(\mathcal A)$, with the normalization chosen so that $\HH^k$ agrees with $k$-dimensional Lebesgue measure on every $k$-dimensional affine subspace.

The next lemma bounds the Hausdorff covering cost of a source ellipsoid by the measure of a corresponding set.

\begin{lemma}\label{ellipsoidcover}
Let $\mathcal E$ be an ellipsoid with semiaxes $b_1\ge\dots\ge b_n$, and let $\delta\ge0$. Take
\[
\cc=\prod_{i<n}b_i.
\]

\begin{itemize}
\item[(i)] The geometric covering estimate is
\[
\HH^{n-1+\delta}_{2b_n}(\mathcal E)\le C\cc\,b_n^\delta.
\]
\item[(ii)] If $\mathcal E$ is the source ellipsoid associated with a section of $v$ at height $h$, with normalizing matrix $T$ and $\beta=\|T\|$, then $b_n=h/\beta\le C\sqrt h$. If a measure $\mu$ and a set $Q$ satisfy $\mu(Q)\ge ch^\epsilon\cc$, then
\begin{equation}\label{coveringcost}
\HH^{n-1+\delta}_{C\sqrt h}(\mathcal E) \le Ch^{\delta/2-\epsilon}\mu(Q).
\end{equation}
\end{itemize}

\end{lemma}

\begin{proof}
A grid of mesh $b_n/\sqrt n$ in the principal frame covers $\mathcal E$ by at most $C\prod_{i<n}(b_i/b_n)$ balls of radius $b_n$. Thus
\[
\HH^{n-1+\delta}_{2b_n}(\mathcal E)\le C\prod_{i<n}(b_i/b_n)b_n^{n-1+\delta}=C\cc b_n^\delta.
\]
For a section of $v$, \eqref{determinant} gives
\[
\beta^n\ge|\det T|\ge ch^{n/2},
\]
so $b_n=h/\beta\le C\sqrt h$. If $\mu(Q)\ge ch^\epsilon\cc$, then
\[
\HH^{n-1+\delta}_{C\sqrt h}(\mathcal E)\le C\cc\,h^{\delta/2}\le Ch^{\delta/2-\epsilon}\mu(Q).
\]
\end{proof}

\subsection{The normalized mass}\label{normalizedmass}
{~}

Let $x\in\mathcal P$, $y=Du(x)$, and let $b_h+T_hB_1$ be a normalizing ellipsoid of $S_h[v](y,x)$. Let $\tilde{v}$ be the normalized potential in \eqref{normalizedpotential} for this section, and let $\tilde{u}=\tilde{v}^*$. John's normalization~\eqref{normalization}, \eqref{raybound} and \eqref{sourcegrowth} verify the selection hypotheses for $\tilde{v}$ and $\tilde{u}$.
With $\rho$ fixed by \eqref{rhochoice}, set
\begin{equation}\label{ellipsoids}
\Omega_h^*(x)=b_h+\rho T_hB_1, \qquad \Omega_h(x)=x+\rho hT_h^{-\mathsf{T}}B_1, \qquad \cc(x,h)=\frac{\|T_h\|h^{n-1}}{|\det T_h|}.
\end{equation}
By \eqref{rhochoice}, Lemma~\ref{selectionlemma} and its source-side version apply to the prescribed sets $\Omega_h^*(x)$ and $\Omega_h(x)$, respectively. Let $\nu$ be a finite Borel measure on $\Omega_1^*$ and $\mu$ a finite Borel measure on $\Omega_0$. Define the \emph{normalized mass} at $x$ of the scale $h$ by
\begin{equation}\label{massratio}
\Theta_h(x)=\frac{\nu(\Omega_h^*(x))+\mu(\Omega_h(x))} {h^\epsilon\cc(x,h)}.
\end{equation}
Throughout Section~\ref{planar} we will take $\nu=\nu_0$ and $\mu=\mu_0$. In Section~\ref{higherdimensions} we keep $\nu_0$, replace $\mu_0$ by a graph projection, and add one term from the graph itself.

Consider the following partition of $\mathcal P$:
\begin{equation}\label{masspartition}
\mathcal P_+=\{x\in\mathcal P:\limsup_{h\downarrow0}\Theta_h(x)>0\} \quad\text{and}\quad \mathcal P_0=\{x\in\mathcal P:\lim_{h\downarrow0}\Theta_h(x)=0\}.
\end{equation}

\begin{proposition}\label{positivemass}
For the finite measures in \eqref{massratio}, $\dim_{\mathrm H}\mathcal P_+\le n-1+2\epsilon$.
\end{proposition}

\begin{proof}
Fix $\delta>2\epsilon$. For every $x\in\mathcal P_+$, there is a positive rational $b$ such that $\Theta_h(x)>2b$ at arbitrarily small heights. At each such height, at least one of the following inequalities holds:
\begin{equation}\label{massalternatives}
\nu(\Omega_h^*(x))\ge b h^\epsilon\cc(x,h) \quad\hbox{or}\quad \mu(\Omega_h(x))\ge b h^\epsilon\cc(x,h).
\end{equation}
Along a sequence of these heights tending to zero, one inequality holds infinitely often. It therefore suffices, for each fixed $b$, to treat separately the points admitting arbitrarily small heights with the first or the second inequality in \eqref{massalternatives}.

Consider the first inequality in \eqref{massalternatives}. For each $k$, we define
\[
I_{2^{-k}}=\bigl\{(x,t):x\in\mathcal P,\ 2^{-k-1}<t\le2^{-k},\ \nu(\Omega_t^*(x))\ge bt^\epsilon\cc(x,t)\bigr\}.
\]
This is the height group $I_h$ in Lemma~\ref{selectionlemma} with $h=2^{-k}$. For $i=(x,t)\in I_{2^{-k}}$, set $x_i=x$, $h_i=t$ and $y_i=Du(x_i)$, and let $b_i+T_iB_1$ normalize $S_{h_i}[v](y_i,x_i)$. Apply Lemma~\ref{selectionlemma} to $v$ with these data, centers $Y_i=T_i^{-1}(b_i-y_i)$, and prescribed sets $Q_i=\Omega_{h_i}^*(x_i)=b_i+\rho T_iB_1$ from \eqref{ellipsoids}. The hypotheses \eqref{localbounds} and \eqref{symmetricgrowth} follow from \eqref{normalization} and \eqref{raybound}. Since $\rho<1/20$, for all sufficiently large $k$, the section containment in \eqref{containment} gives
\[
Q_i=y_i+T_iB_\rho(Y_i) \subset\Omega_1^*\cap\bigl(y_i+T_iB_{1/20}(Y_i)\bigr).
\]
The lemma therefore selects a countable $J\subset I_{2^{-k}}$. The subfamily $(Q_j)_{j\in J}$ has total $\nu$-mass bounded by \eqref{selectedmass}, and \eqref{selectedcover} covers all qualifying source points. Denote these ellipsoids by $\mathcal E_j=x_j+C_*h_jT_j^{-\mathsf{T}}B_1$. By \eqref{coveringcost},
\begin{equation}\label{coveringsum}
\sum_j\HH^{n-1+\delta}_{C2^{-k/2}}(\mathcal E_j) \le C_b\sum_j h_j^{\delta/2-\epsilon}\nu(Q_j) \le C_b2^{-k(\delta/2-\epsilon)}\sum_j\nu(Q_j) \le C_b'2^{-k(\delta/2-\epsilon)}.
\end{equation}
Here $\delta/2-\epsilon>0$. Let $\mathcal P_{+,b}$ be the set of points $x\in\mathcal P$ such that the first inequality in \eqref{massalternatives} holds infinitely often along a sequence of heights tending to zero. For each sufficiently large $k$, denote the source ellipsoids selected above by $\mathcal E_{k,j}$. Then for every sufficiently large integer $k_0$ there holds
\[
\mathcal P_{+,b}\subset\bigcup_{k\geq k_0}\bigcup_{j}\mathcal E_{k,j}.
\]
It follows that
\[
\mathcal H^{n-1+\delta}_{C2^{-k_0/2}}(\mathcal P_{+,b})\leq\sum_{k\geq k_0}\sum_j\mathcal H^{n-1+\delta}_{C2^{-k/2}}(\mathcal E_{k,j})\leq C_b'\sum_{k\geq k_0}2^{-k(\delta/2-\epsilon)},
\]
where \eqref{coveringsum} is used in the last inequality. Since $k_0$ can be arbitrarily large, $\mathcal H^{n-1+\delta}(\mathcal P_{+,b})=0$.

For the second inequality in \eqref{massalternatives}, apply the lemma to $v^*$ with the dual data in \eqref{sourcenormalization}, measure $\mu$, and prescribed sets $Q_i=\Omega_{h_i}(x_i)$ from \eqref{ellipsoids}. Estimate~\eqref{sourcegrowth} verifies the symmetric growth hypothesis. The local boundedness radius $1/(2n)$ in \eqref{s4:e1} and \eqref{rhochoice} verify the prescribed-set condition. The same selection and covering argument applies. Taking the union of the two cases over positive rational $b$ gives $\mathcal H^{n-1+\delta}(\mathcal P_+)=0$. Letting $\delta\downarrow2\epsilon$ proves the conclusion.
\end{proof}

\section{The two-dimensional case}\label{planar}

Throughout this section $n=2$, and $\Theta_h$ uses the weighted measures in \eqref{weightedmeasures}, which are finite by \eqref{finitemeasures}. We estimate the product of their masses on $\Omega_h^*(x)$ and $\Omega_h(x)$.

\subsection{Limits of normalized sections}\label{planarlimit}
{~}

Let $x_j\in\mathcal P$, $y_j=Du(x_j)$ and $h_j\downarrow0$. Write
\begin{equation}\label{potentials}
\tilde{v}_j(Y)=\frac{v(y_j+T_jY)-v(y_j)-x_j\cdot T_jY}{h_j} \quad\text{and}\quad \tilde{u}_j=\tilde{v}_j^*,
\end{equation}
where $T_j$ normalizes $S_{h_j}[v](y_j,x_j)$. Let $b_j+T_jB_1$ be the corresponding normalizing ellipsoid and set $Y_j^0:=T_j^{-1}(b_j-y_j)$. Then $|Y_j^0|\le n$ and
\begin{equation}\label{sectionnormalization}
\tilde{v}_j\ge0,\qquad \tilde{v}_j(0)=0,\qquad B_1(Y_j^0)\subset\{\tilde{v}_j<1\}\subset B_{2n}.
\end{equation}
Choose orthonormal bases in the domain and range of $T_j$ in which $T_j$ is diagonal, and pass to a subsequence on which these bases converge. Put
\[
\beta_j=a_{1,j}\ge a_{2,j}>0,\qquad C_{\mathrm{cov},j}=\frac{\beta_jh_j}{|\det T_j|},\qquad F_j(Y)=F(y_j+T_jY).
\]
By Lemma~\ref{taylorlimit}, after a subsequence there are positive numbers $\gamma_j$ and a nonzero polynomial $P$ of degree at most $d_0$ such that
\begin{equation}\label{polynomiallimit}
P_j:=F_j/\gamma_j\to P\quad\hbox{in }C^2 \hbox{ on bounded sets} \quad\text{and}\quad P(0)=0.
\end{equation}

\begin{proposition}\label{planarclassification}
Let $n=2$ and suppose $\Theta_{h_j}(x_j)\to0$. After passing to a subsequence,
\[
T_j/\beta_j\to e^*\otimes e \quad\text{and}\quad P(Y)=p(e\cdot Y),
\]
where $|e^*|=|e|=1$ and $p$ is a nonzero polynomial of degree at most $d_0$. For the fixed small $\rho$ chosen in \eqref{rhochoice}, there is a convex function $\tilde{u}$ on $B_{2\rho}$ such that $\tilde{u}_j\to \tilde{u}$ locally uniformly there. The subgradients of $\tilde{u}_j$ are uniformly bounded on the smaller ball $B_\rho$. 
\end{proposition}

\begin{proof}
\emph{Step 1. Compactness.} By \eqref{s4:e1}, $0\le\tilde{u}_j\le1$ on $B_{1/(2n)}$. By \eqref{rhochoice}, $4\rho<1/(2n)$. The convex subgradient estimate and Arzel\`a-Ascoli yield a subsequence with $\tilde{u}_j\to\tilde{u}$ locally uniformly on $B_{2\rho}$ and uniformly bounded subgradients on this ball.

\medskip\noindent \emph{Step 2. Rank one limit.} By \eqref{polynomiallimit}, $|F_j|\le C\gamma_j$ on each fixed ball. Since the normalized target ellipsoids have bounded centers and fixed radius $\rho$, and $\{F=0\}$ has Lebesgue measure zero,
\begin{equation}\label{targetmass}
\nu_0(\Omega_{h_j}^*(x_j)) \ge c\gamma_j^{-1+\epsilon}|\det T_j|.
\end{equation}
For almost every $X\in B_\rho$, put $x=x_j+h_jT_j^{-\mathsf{T}}X$. The source normalization \eqref{sourcenormalization} gives
\begin{equation}\label{p5.1:tt}
Du(x)=y_j+T_jD\tilde{u}_j(X).
\end{equation}
By the definition of $F_j$, we have $F(Du(x))=F_j(D\tilde{u}_j(X))$.
So \eqref{polynomiallimit} shows $|F(Du(x))|\le C\gamma_j$. 
Changing variables from $x$ to $X$, we obtain
\begin{equation}\label{sourcemass}
\mu_0(\Omega_{h_j}(x_j)) \ge c\gamma_j^{-1+\epsilon}\frac{h_j^2}{|\det T_j|}.
\end{equation}
Since $C_{\mathrm{cov},j}=h_j/a_{2,j}$, multiplying \eqref{targetmass} and \eqref{sourcemass} gives
\begin{equation}\label{massproduct}
\begin{aligned}
c\left(\frac{\gamma_j}{h_j}\right)^{2\epsilon} \frac{a_{2,j}^2}{\gamma_j^2} \le \frac{\nu_0(\Omega_{h_j}^*(x_j))\mu_0(\Omega_{h_j}(x_j))} {h_j^{2\epsilon}C_{\mathrm{cov},j}^2} \le\tfrac14\Theta_{h_j}(x_j)^2\to 0.
\end{aligned}
\end{equation}
By \eqref{amplitude} and \eqref{axisbounds}, $\gamma_j\ge ch_j$ for $h_j\le1$. Together with $\gamma_j\le C\beta_j$, this proves
\[
\frac{a_{2,j}}{\gamma_j}\to0 \quad\text{and}\quad \frac{a_{2,j}}{\beta_j}\to0.
\]
Choose orthonormal bases $e_{1,j},e_{2,j}$ and $e_{1,j}^*,e_{2,j}^*$ such that $T_je_{i,j}=a_{i,j}e_{i,j}^*$ for $i=1,2$. After passing to a further subsequence, both bases converge. Define
\[
e=\lim_j e_{1,j},\qquad e_2=\lim_j e_{2,j},\qquad e^*=\lim_j e_{1,j}^*.
\]
Then $e,e_2$ are orthonormal, $|e^*|=1$, and
\[
\frac{T_j}{\beta_j} =e_{1,j}^*\otimes e_{1,j} +\frac{a_{2,j}}{\beta_j}e_{2,j}^*\otimes e_{2,j} \to e^*\otimes e.
\]
\medskip\noindent \emph{Step 3. Boundary polynomial.} On each fixed ball, the chain rule gives
\[
|DP_j(Y)\cdot e_{2,j}| =\frac{|DF(y_j+T_jY)\cdot T_je_{2,j}|}{\gamma_j} \le C\frac{a_{2,j}}{\gamma_j}\to0.
\]
Passing to the limit yields $DP\cdot e_2=0$. Thus $P$ is constant along lines parallel to $e_2$, and
\[
P(Y)=p(e\cdot Y) \quad\text{and}\quad p(t):=P(te)\quad(t\in\R).
\]
Since $P$ is a nonzero polynomial of degree at most $d_0$, so is $p$.
\end{proof}

\subsection{Comparison of radii}\label{comparison}
{~}

We next compare the radii of sections centered at nearby points. The argument works in every dimension and will also be used in Section~\ref{higherdimensions}.

\begin{lemma}\label{radiuscomparison}
Let $\tilde{v}\ge0$ be convex, $\tilde{v}(0)=0$, and suppose $B_{r_1}(Y^\circ)\subset\{\tilde{v}<1\}\subset B_{R_1}$ for fixed $r_1,R_1>0$. Let $T:\R^n\to\R^n$ be an invertible linear map and set $\beta=\|T\|$. There are $c,C>0$, depending only on $r_1,R_1$, such that
\[
c\beta\le \sup_{Y\in\overline{S_1[\tilde{v}](Y',X)}}|T(Y-Y')| \le C\beta,\quad\forall X\in B_{1/(4R_1)}, \ \forall Y'\in\partial \tilde{v}^*(X).
\]
\end{lemma}

\begin{proof}
Take $X\in B_{1/(4R_1)}$ and $Y'\in\partial\tilde{v}^*(X)$.
By duality, $X\in\partial\tilde{v}(Y')$ and $\tilde{v}(Y')\le X\cdot Y'$. 
It follows from convexity that
\begin{equation}\label{l5.2:e1}
\tilde{v}(Y)\ge |Y|/R_1\qquad\text{for}\ |Y|\ge R_1.
\end{equation}
Therefore $|Y'|<R_1$. Otherwise the preceding estimates imply $|X|\ge R_1^{-1}$, a contradiction. 

For $Y\in\overline{S_1[\tilde{v}](Y',X)}$ with $|Y|\ge R_1$, we infer from \eqref{l5.2:e1} and $|X|\le (4R_1)^{-1}$ that
\[
1\ge\tilde{v}(Y)-\tilde{v}(Y')-X\cdot(Y-Y')\ge\tilde{v}(Y)-X\cdot Y\ge\frac{3|Y|}{4R_1}.
\]
Hence $S_1[\tilde{v}](Y',X)\subset B_{4R_1/3}$. As a result, for every $Y\in S_1[\tilde{v}](Y',X)$,
\[
|T(Y-Y')|\le\beta(|Y|+|Y'|)\le C_{R_1}\beta.
\] 

Let $e$ be a unit vector such that $|Te|=\beta$. 
Then $|TZ|\ge\beta|e\cdot Z|$ for every $Z\in\mathbb R^n$. Choose $Y_1\in B_{r_1}(Y^\circ)$ such that $|e\cdot(Y_1-Y')|\ge r_1/2$. Since $Y_1,Y'\in B_{R_1}$,
\[
\tilde{v}(Y_1)-\tilde{v}(Y')-X\cdot(Y_1-Y')\le 1+|X||Y_1-Y'| \le 1+\frac{2R_1}{4R_1}=\frac32.
\]
By convexity, $(Y'+Y_1)/2\in S_1[\tilde{v}](Y',X)$. The lower bound follows, since
\[
\left|T\left(\frac{Y'+Y_1}{2}-Y'\right)\right|\ge\frac\beta2|e\cdot(Y_1-Y')|\ge\frac{r_1\beta}{4}.
\]
\end{proof}

With $\rho$ fixed by \eqref{rhochoice}, we now choose $0<c_*\le\rho/4$ in \eqref{radii} and fix it throughout the paper. 

If $T$ normalizes the section at $(Du(x),x)$,
for every $x'\in B_{r_h(x)}(x)$, write
\[
x'=x+\frac h\beta\widehat X \quad\text{and}\quad X=\frac{T^{\mathsf{T}}\widehat X}{\beta},
\]
By \eqref{radii} and the inequality $\beta\le\mathcal R_h(x)$ proved in Lemma~\ref{sectionbounds},
\begin{equation}\label{s5.2:et}
 |X|\le|\widehat X|\le\frac{\beta r_h(x)}h
=c_*\frac{\beta}{\mathcal R_h(x)}\le c_*\le\rho/4.   
\end{equation}
Thus $X$ lies in a fixed ball compactly contained in $B_\rho$. 

Let $\tilde{v}$ be the potential normalized as in \eqref{normalizedpotential} at $(Du(x),x)$. 
For $x'\in\mathcal P\cap B_{r_h(x)}(x)$, let
\[
Y':=T^{-1}(Du(x')-Du(x)).
\] 
By \eqref{sourcenormalization}, $Y'\in\partial\tilde{v}^*(X)$. 
Since $y=Du(x)+TY$, the recentered sections satisfy
\[
S_h[v](Du(x'),x')=Du(x)+T S_1[\tilde{v}](Y',X).
\]
Hence Lemma~\ref{radiuscomparison}, applied to $\tilde{v}$ with inner radius $1$ and outer radius $2n$, shows
\[
c\beta\le\mathcal R_h(x')=\sup_{Y\in\overline{S_1[\tilde{v}](Y',X)}}|T(Y-Y')|\le C\beta.
\]
Since $\beta\le\mathcal R_h(x)\le2n\beta$ by \eqref{axisbounds} and $r_h=c_*h/\mathcal R_h$ by \eqref{radii}, we obtain
\begin{equation}\label{sameheight}
c_2r_h(x)\le r_h(x')\le C_2r_h(x) \qquad(x'\in\mathcal P\cap B_{r_h(x)}(x)).
\end{equation}

By Lemma~\ref{sectionbounds}, $h\mapsto r_h(x)$ is continuous and nondecreasing, and
\begin{equation}\label{shrinkage}
r_h(x)\le C_0h^{1/N}.
\end{equation}
In particular, these radii tend to zero uniformly on $\mathcal P$.

The next proposition converts the sequential limit classification into uniform small-scale radius estimates, without choosing a limit model for each base point and height.

\begin{proposition}\label{refinement}
Let $n\ge2$. Assume that whenever $x_j\in\mathcal P$, $h_j\downarrow0$ and $\Theta_{h_j}(x_j)\to0$, there is a subsequence for which
\begin{equation}\label{ranklimit}
\tilde{u}_j\to \tilde{u}\quad\text{locally uniformly on }B_{2\rho}, \qquad T_j/\beta_j\to e^*\otimes e, \qquad |e^*|=|e|=1,
\end{equation}
and one of the following alternatives holds:

\begin{enumerate}
\item The normalized boundary functions $F_j/\gamma_j$ converge locally uniformly to $P(Y)=p(e\cdot Y)$, where $p$ is nonzero and has degree $\leq d_0$.
\item $\#\{D_e\tilde{u}(X):X\in B_\rho,\ \tilde{u}\text{ is differentiable at }X\}\le d_0$.
\end{enumerate}

Then for each sufficiently small $\theta>0$ and each $\tau\in(0,1)$, 
there are positive thresholds $h_*=h_*(\theta,\tau)$
and $\Theta_*=\Theta_*(\theta,\tau)$ such that, whenever $x\in\mathcal P$, $h<h_*$ and $\Theta_h(x)<\Theta_*$, there is an exceptional set $\mathcal E_{x,h}\subset\mathcal P\cap B_{r_h(x)}(x)$, satisfying the following properties:

\begin{itemize}
\item[(i)] Every subset $\mathcal A\subset\mathcal E_{x,h}$ can be covered by at most $C\theta^{-(n-1)}$ balls of radius $\theta r_h(x)$ with centers in $\mathcal A$.
\item[(ii)] For every $x'\in(\mathcal P\cap B_{r_h(x)}(x))\setminus\mathcal E_{x,h}$,
\begin{equation}\label{refinedradius}
c\theta r_h(x)\le r_{\tau h}(x')\le Cr_h(x).
\end{equation}
\end{itemize}

The choice of $\mathcal E_{x,h}$ also depends on the fixed $\theta,\tau$. 
The constants $c,C$ depend only on $n,d_0$ and the fixed $\rho,c_*$,
but are independent of $\theta,\tau,x,h$ and $\mathcal A$.
\end{proposition}

\begin{proof}
Fix $\theta,\tau$. Assume no thresholds $h_*,\Theta_*$ exist.
Choose $x_j\in\mathcal P$ and $h_j\downarrow0$, with $\Theta_{h_j}(x_j)<1/j$, such that no exceptional set satisfies both (i) and (ii).
By the assumed sequential property, pass to a subsequence satisfying \eqref{ranklimit} and alternative (1) or (2).
We construct exceptional sets $\mathcal E_{x_j,h_j}$ satisfying both conclusions for all large $j$, 
obtaining a contradiction.

\smallskip\noindent \emph{Step 1. Exceptional hyperplanes.} 
In case (1), let $\mathcal T$ be the set of real roots of $p$. In case (2), 
\[
\mathcal T:=\{e\cdot D\tilde{u}(X):X\in B_\rho,\ \tilde{u}\text{ is differentiable at }X\}.
\]
In either case, $\#\mathcal T\le d_0$.
Set $\mathcal I=(-\rho/2,\rho/2)$ and $\psi(t)=\tilde{u}(te)$ for $t\in\mathcal I$.

Take any sequence $x'_j\in\mathcal P\cap B_{r_{h_j}(x_j)}(x_j)$ and define
\[
\widehat X_j:=\frac{\beta_j}{h_j}(x'_j-x_j)\quad\text{and}\quad X_j:=\frac{T_j^{\mathsf{T}}\widehat X_j}{\beta_j}.
\]
By \eqref{s5.2:et}, $|X_j|\le|\widehat X_j|\le\rho/4$. 
After passing to a subsequence, $\widehat X_j\to\widehat X$, and \eqref{ranklimit} implies
\[
X_j\to(e^*\cdot\widehat X)e.
\]
Since $|\widehat X|\le\rho/4$, we have $e^*\cdot\widehat X\in[-\rho/4,\rho/4]\Subset\mathcal I$. We claim that
\begin{equation}\label{contactslopes}
\partial\psi(e^*\cdot\widehat X)\cap\mathcal T\ne\varnothing.
\end{equation}

In case (1), set $Y_j:=T_j^{-1}(Du(x'_j)-y_j)$. By \eqref{sourcenormalization}, $Y_j\in\partial\tilde{u}_j(X_j)$. 
Local uniform convergence controls these subgradients, so a subsequence satisfies
\[
Y_j\to Y\in\partial \tilde{u}((e^*\cdot\widehat X)e).
\]
Since $\mathcal P\subset E$, these target points satisfy $F_j(Y_j)=0$. 
By \eqref{polynomiallimit}, $p(e\cdot Y)=0$. 
Restricting the supporting inequality for $\tilde{u}$ to the line $te$ gives $e\cdot Y\in\partial\psi(e^*\cdot\widehat X)$, proving \eqref{contactslopes}.

In case (2), for each $t\in\mathcal I$, approach $te$ by differentiability points of $\tilde{u}$ in $B_\rho$. 
A subsequence of the gradients converges to $q\in\partial\tilde{u}(te)$, with $e\cdot q\in\mathcal T$.
Restricting the supporting inequality to the line $\mathbb R e$ gives $\partial\psi(t)\cap\mathcal T\ne\varnothing$. 
Taking $t=e^*\cdot\widehat X$ proves \eqref{contactslopes}.

For $s\in\mathcal T$, consider
\[
\mathcal I_s=\{t\in\mathcal I:s\in\partial\psi(t)\}.
\]
Each nonempty $\mathcal I_s$ is a relatively closed interval on which $\psi$ is affine with slope $s$.
Denote the endpoints of these intervals lying in $\mathcal I$ by $t_i$. 
The intervals $\mathcal I_s$ need not meet.
Since $\mathcal T$ contains at most $d_0$ elements, there are at most $2d_0$ such endpoints. 
By \eqref{contactslopes}, every limiting position lies over one of the intervals $\mathcal I_s$. 
Thus, away from the hyperplanes $\{\widehat X:e^*\cdot\widehat X=t_i\}$, it lies in the interior of an interval on which $\psi$ is affine.

\smallskip\noindent \emph{Step 2. Exceptional covering.} Take $\vartheta$ to be a sufficiently small fixed multiple of $\theta$. Define
\[
\mathcal E_{x_j,h_j} =\Bigl\{x'\in\mathcal P\cap B_{r_{h_j}(x_j)}(x_j): \Big|\frac{\beta_j}{h_j}e^*\cdot(x'-x_j)-t_i\Big| 
\le2\vartheta\text{ for some }i\Bigr\}.
\]
By \eqref{radii} and \eqref{axisbounds}, $h_j/\beta_j$ is comparable to $r_{h_j}(x_j)$. 
Thus $\mathcal E_{x_j,h_j}$ lies in at most $2d_0$ slabs of width $C\vartheta r_{h_j}(x_j)$. 
For each slab, its intersection with $B_{r_{h_j}(x_j)}(x_j)$ can be covered by at most $C\vartheta^{-(n-1)}$ balls of radius $C\vartheta r_{h_j}(x_j)$. 
Choose the fixed multiple so that this radius is at most $\theta r_{h_j}(x_j)/2$. 
Given $\mathcal A\subset\mathcal E_{x_j,h_j}$, retain only those balls in the preceding cover that meet $\mathcal A$. For each retained ball, choose a point of $\mathcal A$ inside it as the new center. Doubling the radii gives (i).

\smallskip\noindent \emph{Step 3. Radius estimates.} 
The upper bound follows from $0<\tau<1$ and \eqref{sameheight}:
\[
r_{\tau h_j}(x')\le r_{h_j}(x')\le C_2r_{h_j}(x_j).
\]
For the lower bound, take a sequence of points outside $\mathcal E_{x_j,h_j}$. 
Pass to a subsequence with normalized positions $\widehat X_j\to\widehat X$, and set $t_0=e^*\cdot\widehat X$. 
Step 1 and the definition of $\mathcal E_{x_j,h_j}$ imply that $\psi$ is affine on $[t_0-\vartheta,t_0+\vartheta]$. 
The corresponding target points satisfy, after a subsequence,
\[
X_j=\frac{T_j^{\mathsf{T}}\widehat X_j}{\beta_j}\to t_0e, \qquad Y_j\in\partial \tilde{u}_j(X_j), \qquad Y_j\to Y_0\in\partial \tilde{u}(t_0e).
\]
Since $\psi$ is affine on this interval, $e\cdot Y_0=\psi'(t_0)$. Hence the nonnegative quantities
\[
d_j^\pm=\tilde{u}_j(X_j\pm\vartheta e)-\tilde{u}_j(X_j) \mp\vartheta e\cdot Y_j\to0.
\]

The normalized target section at $x'$ is $S_\tau[\tilde{v}_j](Y_j,X_j)$. 
Since $\tilde{u}_j=\tilde{v}_j^*$ and $Y_j\in\partial\tilde{u}_j(X_j)$,
\[
\tilde{v}_j(Y)\ge(X_j\pm\vartheta e)\cdot Y-\tilde{u}_j(X_j\pm\vartheta e)\quad\text{and}\quad
\tilde{v}_j(Y_j)=X_j\cdot Y_j-\tilde{u}_j(X_j).
\]
Subtracting the equality from the inequality and rearranging with the definition of $d_j^\pm$ gives
\[
\tilde{v}_j(Y)-\tilde{v}_j(Y_j)-X_j\cdot(Y-Y_j)\ge\pm\vartheta e\cdot(Y-Y_j)-d_j^\pm.
\]
Hence
\begin{equation}\label{sectionwidth}
Y\in S_\tau[\tilde{v}_j](Y_j,X_j)\quad\Longrightarrow\quad |e\cdot(Y-Y_j)|\le\frac{\tau+\max(d_j^+,d_j^-)}{\vartheta}.
\end{equation}
The proof of Lemma~\ref{radiuscomparison} bounds these sections in a fixed ball. 
Using \eqref{sectionwidth} and \eqref{ranklimit},
\[
\frac{\mathcal R_{\tau h_j}(x')}{\beta_j} =\sup_{Y\in S_\tau[\tilde{v}_j](Y_j,X_j)} \left|\frac{T_j}{\beta_j}(Y-Y_j)\right| \le\frac\tau\vartheta+o(1) \le\frac{2\tau}{\vartheta}
\]
for large $j$, with $\vartheta,\tau$ fixed. By \eqref{radii} and \eqref{axisbounds}, $r_{h_j}(x_j)\le c_*h_j/\beta_j$. Hence
\begin{equation}\label{affineradius}
r_{\tau h_j}(x') =\frac{c_*\tau h_j}{\mathcal R_{\tau h_j}(x')} \ge\frac{c_*\vartheta h_j}{2\beta_j} \ge\frac{\vartheta}{2}r_{h_j}(x_j).
\end{equation}
This estimate holds uniformly outside $\mathcal E_{x_j,h_j}$ for large $j$. Otherwise, a sequence of violating points would have a convergent subsequence to which the preceding argument applies. Since $\vartheta$ is a fixed multiple of $\theta$, this proves (ii). The constants in (i) and (ii) are independent of $\theta,\tau$. Both conclusions now hold for large $j$, contradicting the choice of the sequence.
\end{proof}

\subsection{The covering argument}\label{covering}
{~}

\begin{proposition}\label{zeromass}
Under the hypotheses of Proposition~\ref{refinement}, $\dim_{\mathrm H}\mathcal P_0\le n-1$.
\end{proposition}

\begin{proof}
\smallskip\noindent \emph{Step 1. Covering contraction.} Fix $\delta>0$ and put
\[
s=n-1+\delta,\qquad d=n-1+\delta/2,\qquad a=\delta/(2N).
\]
Choose $\theta<c_2$ so small that $C\theta^{\delta/2}<1/4$, and then choose $\tau>0$ so small that $C\theta^{-n}\tau^a<1/4$. The constants here include those in the estimates below. With these $\theta,\tau$, take $h_* =h_*(\theta,\tau)$ and $\Theta_*=\Theta_*(\theta,\tau)$ 
from Proposition~\ref{refinement}. Fix $m$ and write
\[
E_m=\{x\in\mathcal P_0:\Theta_h(x)<\Theta_* \text{ for every }0<h<1/m\} \quad\text{and}\quad \bar h=\min\{h_*,1/m\}.
\]
These sets cover $\mathcal P_0$. Fix $x\in E_m$ and $h\in (0,\bar h)$.

For each $E_m$, we consider the following decompositions
\[
\mathcal A=\{x'\in E_m\cap B_{r_h(x)}(x):r_{\tau h}(x')\notin[c\theta r_h(x),Cr_h(x)]\} \quad\text{and}\quad \mathcal G=(E_m\cap B_{r_h(x)}(x))\setminus\mathcal A.
\]
By Proposition~\ref{refinement} (ii), $\mathcal A\subset\mathcal E_{x,h}$. 
According to Proposition~\ref{refinement} (i)
\[
\mathcal A\subset \bigcup_{i\in I} B_{\theta r_h(x)}(x_i),\quad\text{with} \ x_i\in\mathcal A \ \text{and}\ \# I\le C\theta^{-(n-1)}. 
\] 
For $\mathcal G$, choose a maximal $c\theta r_h(x)/2$-separated family of centers $\{x_j\}_{j\in J}\subset\mathcal G$. The balls of radius $c\theta r_h(x)/4$ about them are disjoint in $B_{2r_h(x)}(x)$, so there are at most $C\theta^{-n}$ centers. Since $r_{\tau h}(x_j)\ge c\theta r_h(x)$ by the definition of $\mathcal G$, maximality of the separated family then ensures
\[
\mathcal G\subset\bigcup_{j\in J}B_{r_{\tau h}(x_j)}(x_j),\quad\text{with} \ x_j\in\mathcal G, \ \# J\le C\theta^{-n}\
\text{and}\ r_{\tau h}(x_j)\le Cr_h(x).
\]

Together these balls cover $E_m\cap B_{r_h(x)}(x)$. 
We also have
\begin{equation}\label{contraction}
\begin{aligned}
\sum_{i\in I} (\theta r_h(x))^d h^a+\sum_{j\in J} r_{\tau h}(x_j)^d (\tau h)^a
&\le C\theta^{-(n-1)}(\theta r_h(x))^dh^a +C\theta^{-n}r_h(x)^d(\tau h)^a\\
&\le C\bigl(\theta^{\delta/2}+\theta^{-n}\tau^a\bigr)r_h(x)^d h^a \le\frac12 r_h(x)^d h^a.
\end{aligned}
\end{equation}
\smallskip\noindent \emph{Step 2. Hausdorff content.} Set
\[
M=\sup_{\substack{x\in E_m\\
0<h<\bar h}} \frac{\HH^s_\infty(E_m\cap B_{r_h(x)}(x))}{r_h(x)^d h^a},
\]
with $M=0$ if $E_m$ is empty. This supremum is finite, since \eqref{shrinkage} gives
\[
\HH^s_\infty(E_m\cap B_{r_h(x)}(x))\le C_sr_h(x)^s \le C r_h(x)^d h^a.
\]
All covering centers lie in $E_m$. For each $i\in I$, \eqref{sameheight} gives $r_h(x_i)\ge c_2r_h(x)>\theta r_h(x)$. Continuity in the height and $r_t(x_i)\to0$ as $t\downarrow0$ give a height $t\in(0,h]$ with $r_t(x_i)=\theta r_h(x)$. Since $t\le h<\bar h$, the definition of $M$ gives
\[
\HH^s_\infty(E_m\cap B_{\theta r_h(x)}(x_i))\le M(\theta r_h(x))^d t^a\le M(\theta r_h(x))^d h^a.
\]
For each $j\in J$, the corresponding height is $\tau h<\bar h$, so
\[
\HH^s_\infty(E_m\cap B_{r_{\tau h}(x_j)}(x_j))\le M r_{\tau h}(x_j)^d(\tau h)^a.
\]
Apply subadditivity to the cover in Step 1 and use \eqref{contraction}:
\begin{align*}
\HH^s_\infty(E_m\cap B_{r_h(x)}(x)) &\le\sum_{i\in I}\HH^s_\infty(E_m\cap B_{\theta r_h(x)}(x_i)) +\sum_{j\in J}\HH^s_\infty(E_m\cap B_{r_{\tau h}(x_j)}(x_j))\\
&\le M\left(\sum_{i\in I}(\theta r_h(x))^d h^a+\sum_{j\in J}r_{\tau h}(x_j)^d(\tau h)^a\right) \le\frac M2 r_h(x)^d h^a.
\end{align*}
Dividing by $r_h(x)^dh^a$ and taking the supremum over $x\in E_m$ and $0<h<\bar h$ gives $M\le M/2$, hence $M=0$. A countable subcover of $E_m$ by the balls in this supremum shows
\[
\HH^s_\infty(E_m)=0.
\]
Since zero Hausdorff content is equivalent to zero Hausdorff measure, $\HH^s(E_m)=0$. Taking the countable union over $m$ and letting $\delta\downarrow0$ completes the proof.

\end{proof}

\section{Higher dimensions}\label{higherdimensions}

Let $n\ge3$. We retain the localization and fixed sets $\mathcal P$ from Section~\ref{localization}. Lemma~\ref{radiuscomparison} is valid in every dimension. To use Propositions~\ref{refinement} and~\ref{zeromass}, we must establish the limit classification assumed in Proposition~\ref{refinement}. We introduce an additional finite measure on the subgradient graph to estimate the transverse derivatives of the normalized boundary polynomial. Together with the target and source mass estimates, these estimates yield the alternatives in Proposition~\ref{classification}.

For the fixed potential $v$ in \eqref{fixeddual}, let $v_1$ be its Moreau regularization,
\begin{equation}\label{regularization}
v_1(z)=\min_{y'\in\R^n}\bigl\{v(y')+\tfrac12|z-y'|^2\bigr\}, \qquad z\in\R^n.
\end{equation}
The minimizer $y(z)$ is unique. Take $x(z)=z-y(z)$.
By the standard properties of Moreau regularization~\cite{M65}, 
$v_1\in C^{1,1}$, both $v_1$ and $\tfrac12|z|^2-v_1$ are convex, and
\begin{equation}\label{s6:e0}
Dv_1(z)=x(z).
\end{equation}
The mapping $z\mapsto y(z)$ was introduced in \cite{M62}, and can be written as
\begin{equation}\label{s6:e1}
y(z) = z-  Dv_1(z) .
\end{equation}
The maps $x,y$ are Lipschitz continuous, $x(z)\in\partial v(y(z))$ 
and every pair $x\in\partial v(y)$ arises from exactly one $z=x+y$.

Consider $K(z)=D_zy$. Then $K = I -D^2v_1$ almost everywhere. By \eqref{s6:e0} and \eqref{s6:e1}
\begin{equation}\label{graphderivatives}
K=K^{\mathsf{T}},\qquad 0\le K\le I,\qquad D_zx=I-K.
\end{equation}
All graph integrals below are restricted to a fixed compact subset of the graph
\[
\{(y,x)\in\R^n\times\R^n:x\in\partial v(y)\},
\]
with pairs $(y,x)\in\Omega_0^*\times\Omega$, chosen to contain the sets used below. 
The margins in Definition~\ref{preparedset} allow this choice uniformly at small heights.

\subsection{Weighted measures on the graph}\label{graphmeasures}
{~}

Let $\Gamma_0\Subset\Omega_0^*\times\Omega$ be a fixed compact graph subset and set
\[
\mathcal Z_0:=\{x+y:(y,x)\in\Gamma_0\}
=\{z\in\R^n:(y(z),x(z))\in\Gamma_0\}.
\]
Define $\widetilde\nu_\epsilon$ on the $z$-parameter space $\R^n$ by the density below on $\mathcal Z_0$, extended by zero outside $\mathcal Z_0$. Its pushforward $\widetilde\mu_\epsilon$ under $z\mapsto x(z)$ is a measure on the source space $\R^n$:
\begin{equation}\label{graphmeasure}
\begin{gathered}
d\widetilde\nu_\epsilon(z)=\mathbf1_{\{F(y(z))\ne0\}} |F(y(z))|^{-1+\epsilon} DF(y(z))^{\mathsf{T}}K(z)DF(y(z))\,dz\\
\widetilde\mu_\epsilon=x_\#\widetilde\nu_\epsilon.
\end{gathered}
\end{equation}
The finiteness of these measures follows directly by integration by parts on the graph.

\begin{lemma}\label{graphfiniteness}
The measures $\widetilde\nu_\epsilon$ and $\widetilde\mu_\epsilon$ are finite.
\end{lemma}

\begin{proof}
Choose a nonnegative $\chi\in C_c^1(y^{-1}(\Omega_0^*))$ equal to one on $\mathcal Z_0$. 
For $t>0$, consider
\[
b_t(s)=\int_0^s(|r|+t)^{-1+\epsilon}\,dr \quad\text{and}\quad |b_t(s)|\le |s|^\epsilon/\epsilon.
\]
Since $K=Dy$ is symmetric, the Lipschitz chain rule gives, almost everywhere,
\[
\begin{aligned}
D_z[b_t(F(y))]&=(|F(y)|+t)^{-1+\epsilon}KDF(y),\\
\operatorname{div}_z(\chi DF(y))&=D\chi\cdot DF(y)+\chi\,\tr(D^2F(y)K).
\end{aligned}
\]
Integration by parts, using the compact support of $\chi$, therefore gives
\begin{align*}
&\int_{\R^n}\chi\cdot(|F(y)|+t)^{-1+\epsilon} DF(y)^{\mathsf{T}}KDF(y)\,dz\\
&\qquad=-\int_{\R^n} b_t(F(y))\bigl[D\chi\cdot DF(y) +\chi\,\tr(D^2F(y)K)\bigr]\,dz\le C.
\end{align*}
The constant $C$ is independent of $t$, by the bound for $b_t$, \eqref{graphderivatives}, and the bounds for $F,DF,D^2F$ on the fixed compact set $y(\operatorname{supp}\chi)$. Restrict the left side to $\mathcal Z_0\cap\{F\circ y\ne0\}$ and let $t\downarrow0$. Monotone convergence proves that $\widetilde\nu_\epsilon$ is finite. Its pushforward $\widetilde\mu_\epsilon$ is finite as well.
\end{proof}

\subsection{Regularized sections}\label{regularizedsections}
{~}

For a pair $(y_0,x_0)$ set $z_0=x_0+y_0$. 
According to \eqref{regularization}, $v_1(z_0) =v(y_0)+\frac12|x_0|^2$, and
\[
v_1(z_0+z) = \min_y\big\{ v(y_0+y)+\frac12 |x_0+z-y|^2\big\}.
\]
Completing the square then shows that
\begin{equation}\label{regularizedidentity}
v_1(z_0+z)-v_1(z_0)-x_0\cdot z =\min_y\left\{v(y_0+y)-v(y_0)-x_0\cdot y +\tfrac12|z-y|^2\right\}.
\end{equation}
Together with $Dv_1(z_0) =x_0$, see \eqref{s6:e0}, we conclude that
\begin{equation}\label{regularizedinclusions}
\tfrac12\bigl(S_h[v](y_0,x_0)-y_0\bigr)+B_{\sqrt h} \subset S_h[v_1](z_0)-z_0 \subset S_h[v](y_0,x_0)-y_0+B_{\sqrt{2h}}.
\end{equation}
For the first inclusion, write $z=\xi/2+q$ with $\xi\in S_h[v](y_0,x_0)-y_0$ and $|q|<\sqrt h$. 
Taking $y=\xi/2$ in \eqref{regularizedidentity} and using convexity shows
\[
v_1(z_0+z)-v_1(z_0)-x_0\cdot z\le\frac12\bigl(v(y_0+\xi)-v(y_0)-x_0\cdot\xi\bigr)+\frac12|q|^2<h.
\]
On the other side, take $z\in S_h[v_1](z_0)-z_0$ and 
let $y$ be the minimizer in \eqref{regularizedidentity}.
Then
\[
v(y_0+y)-v(y_0)-x_0\cdot y \le h -\tfrac12|z-y|^2.
\]
Hence $y\in S_h[v](y_0,x_0)-y_0$.
Since $v(y_0+y)-v(y_0)-x_0\cdot y\ge0$, we obtain $|z-y|\le \sqrt{2h}$.
The second inclusion follows.

Denote by $b_h+T_hB_1$ the inner normalizing ellipsoid of $S_h[v](y_0,x_0)$.
Suppose $T_h=O_h\operatorname{diag}(a_i)$ for some orthogonal matrix $O_h$. Let
\begin{equation}\label{regularizedmatrix}
a_i'=\sqrt{a_i^2+h},\qquad \widehat T_h=O_h\operatorname{diag}(a_i'),\qquad \beta_h'=\sqrt{\beta(h)^2+h}.
\end{equation}
Then $(\det \widehat T_h)^2=\prod_i(a_i^2+h)$ and $\beta_h'\approx\beta(h)$ by \eqref{determinant}. 
The product of the largest $n-1$ semiaxes of the ellipsoid $x_0+h\widehat T_h^{-\mathsf{T}}B_1$ is
\begin{equation}\label{regularizedcost}
\cc'(x_0,h)=\frac{\beta_h'h^{n-1}}{|\det \widehat T_h|}.
\end{equation}
For $|Z|<1$, take $z=b_h-y_0+\widehat T_hZ$ and $y=b_h-y_0+T_hZ\in S_h[v](y_0,x_0)-y_0$ in \eqref{regularizedidentity}. 
Then
\begin{align}\label{graphinterior}
v_1(x_0+b_h+\widehat T_hZ)-v_1(z_0) -x_0\cdot(b_h-y_0+\widehat T_hZ)\le h+\frac12|(\widehat T_h-T_h)Z|^2
<\frac32h,
\end{align}
since $\|\widehat T_h-T_h\|\le\sqrt h$ in the principal frame. In particular,
\[
x_0+b_h+\widehat T_hB_1\subset S_{4h}[v_1](z_0).
\]
With the fixed $\rho$ from \eqref{rhochoice}, let us set
\begin{equation}\label{graphset}
G_h(x_0)=x_0+b_h+\rho \widehat T_hB_1.
\end{equation}
In the coordinates $z=x_0+b_h+\widehat T_hZ$, the set $G_h(x_0)$ corresponds to $B_\rho$ 
and the image of the section $S_{4h}[v_1](z_0)$ in these coordinates contains $B_1$. Thus $G_h(x_0)$ has distance at least $1-\rho>1/2$ from the section boundary in these coordinates, uniformly in $x_0$ and $h$.

To apply Lemma~\ref{selectionlemma} to $v_1$, 
we use base point $z_0$, supporting slope $x_0$, height $h$, 
linear map $\widehat T_h$ and normalized center
$Z_0:=\widehat T_h^{-1}(b_h-y_0)$. In the coordinates
$z=z_0+\widehat T_hZ$, the corresponding normalized potential is
\begin{equation}\label{normalizedregularization}
\widetilde v_{1,h}(Z)
=\frac{v_1(z_0+\widehat T_hZ)-v_1(z_0)-x_0\cdot\widehat T_hZ}{h}.
\end{equation}
Since $\|\widehat T_h^{-1}T_h\|\le1$ and $S_h[v](y_0,x_0)\subset b_h+nT_hB_1$,
we have
\[
|Z_0|\le |T_h^{-1}(b_h-y_0)|\le n.
\]
Estimate~\eqref{graphinterior} shows $\widetilde v_{1,h}\le3/2$ on $B_1(Z_0)$.
This verifies \eqref{localbounds} for $\widetilde v_{1,h}$.
Set $R:=2n+\sqrt2$. Since $\sqrt h\,\|\widehat T_h^{-1}\|\le1$,
we infer by the second conclusion in \eqref{regularizedinclusions} that
\[
\{\widetilde v_{1,h}<1\}
=\widehat T_h^{-1}\big(S_h[v_1](z_0)-z_0\big)\subset B_R.
\]
For $|Z|\ge R$, take $t=R/|Z|\le1$. Convexity gives $1\le\widetilde v_{1,h}(tZ)\le t\widetilde v_{1,h}(Z)$.
Hence
\[
\widetilde v_{1,h}(Z)\ge |Z|/R-1\qquad(Z\in\R^n).
\]
Consequently, for every $\xi\in\R^n$,
\[
\widetilde v_{1,h}(Z_0+\xi)+\widetilde v_{1,h}(Z_0-\xi)
\ge\frac{|Z_0+\xi|+|Z_0-\xi|}{R}-2
\ge\frac{2|\xi|}{R}-2.
\]
This verifies \eqref{symmetricgrowth} at the prescribed center $Z_0$
with $c=2/(2n+\sqrt2)$ and $C=2$, uniformly in $x_0$ and $h$.
Apply Lemma~\ref{selectionlemma} to $v_1$ with
$Q_i=G_{h_i}(x_i)$. Since $r_1=1$ and
$\rho<1/20$ by \eqref{rhochoice},
\[
Q_i=z_i+\widehat T_iB_\rho(Z_i)
\subset z_i+\widehat T_iB_{r_1/20}(Z_i),
\]
where $z_i=x_i+Du(x_i)$ and
$Z_i=\widehat T_i^{-1}(b_i-Du(x_i))$.
In each height range $h/2<h_i\le h$, the lemma selects
a countable subfamily, indexed by $j$, such that
\[
\sum_j\widetilde\nu_\epsilon(G_{h_j}(x_j)) \le C\widetilde\nu_\epsilon(\R^n)\quad\text{and}\quad \{x_i:h/2<h_i\le h\} \subset\bigcup_j \big(x_j+Ch_j\widehat T_j^{-\mathsf{T}}B_1\big).
\]

Let $\nu_0$ be the measure defined in \eqref{weightedmeasures}.
Using \eqref{regularizedcost} and \eqref{graphset}, we replace \eqref{massratio} by
\begin{equation}\label{graphmassratio}
\Theta_h(x)= \frac{\nu_0(\Omega_h^*(x))+\widetilde\mu_\epsilon(\Omega_h(x))} {h^\epsilon\cc(x,h)} +\frac{\widetilde\nu_\epsilon(G_h(x))} {h^\epsilon\cc'(x,h)}.
\end{equation}
Define $\mathcal P_+$ and $\mathcal P_0$ by \eqref{masspartition} using this function. Since $a_i'\ge\sqrt h$, each semiaxis of $h\widehat T_h^{-\mathsf{T}}B_1$ is at most $\sqrt h$. The finiteness in Lemma~\ref{graphfiniteness}, the selection just described, and Lemma~\ref{ellipsoidcover} therefore yield the same covering estimate for the third term. 
The proof of Proposition~\ref{positivemass}, now with three alternatives instead of two, shows
\[
\dim_{\mathrm H}\mathcal P_+\le n-1+2\epsilon.
\]

\subsection{Limits of normalized sections}\label{higherlimit}
{~}

Let $x_j\in\mathcal P$, $y_j=Du(x_j)$ and $h_j\downarrow0$.
Write $C_{\mathrm{cov},j}=C_{\mathrm{cov}}(x_j,h_j)$ and
$C_{\mathrm{cov},j}'=C_{\mathrm{cov}}'(x_j,h_j)$,
with the covering costs defined in
\eqref{ellipsoids} and \eqref{regularizedcost}, respectively.
Let $b_j+T_jB_1$ be a normalizing ellipsoid of
$S_{h_j}[v](y_j,x_j)$ as in \eqref{normalization}, with ordered
semiaxis lengths $a_{1,j}\ge\cdots\ge a_{n,j}>0$.
Recall that $\beta_j=\|T_j\|=a_{1,j}$.
Denote by $\tilde v_j$ the normalized target potential and by $\tilde u_j$ its Legendre transform defined in \eqref{potentials}.
For all sufficiently large $j$, \eqref{containment} and \eqref{fixeddual}
allow us to use the source normalization \eqref{sourcenormalization}
with $v^*=u$ for $X\in B_{2\rho}$.
Let $F_j(Y)=F(y_j+T_jY)$. Define
\begin{equation}\label{axisratios}
\sigma_{i,j}=a_{i,j}^2/h_j\quad\text{and}\quad \kappa_j=(\det T_j)^2/h_j^n=\prod_{i=1}^n\sigma_{i,j}.
\end{equation}
Thus $\kappa_j$ is the scaling factor for the Monge-Amp\`ere density under the normalization \eqref{potentials}. Choose orthonormal bases in the domain and range of $T_j$ in which $T_j$ is diagonal, and pass to a subsequence on which these bases converge. After passing to a further subsequence, $\tilde{u}_j\to \tilde{u}$ locally uniformly on ${B_{2\rho}}$, by the compactness argument in Step 1 of the proof of Proposition~\ref{planarclassification}. 
By Lemma~\ref{taylorlimit}, after passing to a further subsequence, the convergence~\eqref{polynomiallimit} holds for a nonzero polynomial $P$ of degree at most $d_0$.

\begin{proposition}\label{classification}
Suppose $\Theta_{h_j}(x_j)\to0$, with $\Theta_h$ given by \eqref{graphmassratio}. After a subsequence,
\[
T_j/\beta_j\to e^*\otimes e \quad\text{and}\quad |e^*|=|e|=1.
\]
At least one of the following holds:

\begin{enumerate}
\item $P(Y)=p(e\cdot Y)$ for a nonzero polynomial $p$ of degree at most $d_0$.
\item $\#\{D_e\tilde{u}(X):X\in B_\rho,\ \tilde{u}\text{ is differentiable at }X\}\le d_0$.
\end{enumerate}

\end{proposition}

\begin{proof}
Write $\Theta_j=\Theta_{h_j}(x_j)$. We work in the converging orthonormal bases chosen above.

\smallskip\noindent \emph{Step 1. Source estimate.} 
Estimates \eqref{axisbounds}, \eqref{determinant} and \eqref{amplitude} show 
\begin{equation}\label{p6.2:e0}
\kappa_j\ge c,\quad a_{n,j}\ge ch_j\quad \text{and} \quad ch_j\le\gamma_j\le C\beta_j.
\end{equation}
By \eqref{ellipsoids}, the estimate~\eqref{targetmass}, valid in all dimensions, and \eqref{axisratios}, we obtain
\begin{equation}\label{targetbound}
\frac{\nu_0(\Omega_{h_j}^*(x_j))}{h_j^\epsilon C_{\mathrm{cov},j}}
 \ge c\left(\frac{\gamma_j}{h_j}\right)^\epsilon \frac{\kappa_jh_j}{\beta_j\gamma_j}.
\end{equation}
By \eqref{p6.2:e0}, the right-hand side is at least $ch_j/\beta_j^2=c/\sigma_{1,j}$. 
As $\Theta_{h_j}(x_j)\to0$,   we find that
\begin{equation}\label{largestaxisdivergence}
\sigma_{1,j}\to\infty.
\end{equation}

Almost every $x\in\Omega_{h_j}(x_j)$ belongs to the regular set $\Omega'$ of Lemma~\ref{regularset}. 
At such $x$, we have $D^2u(x)=[D^2v(y)]^{-1}$ with $y=Du(x)$. 
By \eqref{s6:e0} and \eqref{s6:e1}, $z=y+Dv(y)$.
Differentiating this identity and using $K=D_zy$ as in \eqref{graphderivatives}, 
we obtain 
\[
K=(I+D^2v(y))^{-1}\quad\text{and}\quad dz=\det(I+D^2v(y))\,dy.
\] 
Changing variables in~\eqref{graphmeasure} and using~\eqref{dualequation}, we obtain
\begin{equation}\label{sourcegraphmass}
\widetilde\mu_\epsilon(\Omega_{h_j}(x_j))\ge 
\lambda/\Lambda\int_{\Omega_{h_j}(x_j)} F(Du)^{-1+\epsilon}DF(Du)^{\mathsf{T}}\cof(I+D^2v(y))DF(Du)\,dx,
\end{equation}
where $\cof A=(\det A)A^{-1}$ denotes the cofactor matrix. For $M\ge1$, let
\[
\mathcal Q_j(M)=\{X\in{B_\rho}:D^2\tilde{u}_j(X)\le MI\}.
\]
Choose $\chi\in C_c^\infty({B_{2\rho}})$ with $\chi\ge0$ and $\chi=1$ on ${B_\rho}$. The uniform bounds for $\tilde{u}_j$ imply
\begin{equation*}
\int_{{B_\rho}}\tr D^2\tilde{u}_j\,dX \le\int\chi\,d(\Delta \tilde{u}_j)=\int \tilde{u}_j\Delta\chi\,dX\le C.
\end{equation*}
It follows that
\begin{equation}\label{hessianbound}
|B_\rho\setminus\mathcal Q_j(M)|\le C/M.
\end{equation}
For all sufficiently large $j$ and almost every $X\in\mathcal Q_j(M)$, write
\begin{equation}\label{p6.2:e1}
x=x_j+h_jT_j^{-\mathsf{T}}X.
\end{equation}
Then $D^2u(x)\le Mh_j^{-1}T_jT_j^{\mathsf{T}}$. Thus, for $y=Du(x)$ at such $x$,
\[
D^2v(y)\ge M^{-1}h_j(T_jT_j^{\mathsf{T}})^{-1}.
\]
Therefore, in the principal-axis coordinates of the target ellipsoid,
\begin{equation}\label{p6.2:e2}
\cof(I+D^2v(y))\ge c_M\operatorname{diag}\Big(\prod_{i\ne l}(1+\sigma_{i,j}^{-1})\Big)_{l=1}^n.
\end{equation}
At these points, 
\begin{equation}\label{p6.2:e3}
0<F(Du(x))\le C\gamma_j\quad\text{and}\quad dx=h_j^n|\det T_j|^{-1}\,dX.
\end{equation}
Recall that $P_j(Y) = \gamma_j^{-1} F(y_j+T_jY)$ in \eqref{polynomiallimit}. 
Differentiating it and using \eqref{p5.1:tt} and \eqref{p6.2:e1},
\begin{equation}\label{energyidentities}
[DF(Du(x))]_l=\frac{\gamma_j}{a_{l,j}}\partial_lP_j(D\tilde{u}_j(X)).
\end{equation}
Restricting the integral in \eqref{sourcegraphmass} to the points corresponding to $\mathcal Q_j(M)$ and using \eqref{p6.2:e2}, \eqref{p6.2:e3} and \eqref{energyidentities}, we obtain
\begin{align}
\widetilde\mu_\epsilon(\Omega_{h_j}(x_j))\notag
&\ge c_M\gamma_j^{1+\epsilon}\frac{h_j^n}{|\det T_j|}
\int_{\mathcal Q_j(M)}\sum_{l=1}^n a_{l,j}^{-2}
|\partial_lP_j(D\tilde u_j(X))|^2
\prod_{i\ne l}(1+\sigma_{i,j}^{-1})\,dX\notag\\
&\quad= c_M\gamma_j^{1+\epsilon}\frac{C_{\mathrm{cov},j}}{\kappa_j\beta_j}
\int_{\mathcal Q_j(M)}\sum_{l=1}^n
|\partial_lP_j(D\tilde u_j(X))|^2
\prod_{i\ne l}(1+\sigma_{i,j})\,dX.\notag
\end{align}
Dividing by $h_j^\epsilon C_{\mathrm{cov},j}$ and multiplying by \eqref{targetbound}, 
we use $\gamma_j/h_j\ge c$ to obtain
\begin{equation}\label{sourceenergy}
\int_{\mathcal Q_j(M)}\sum_{l=1}^n \frac{\prod_{i\ne l}(1+\sigma_{i,j})}{\sigma_{1,j}} |\partial_lP_j(D\tilde{u}_j)|^2\,dX\le C_M\Theta_j^2.
\end{equation}

\smallskip\noindent \emph{Step 2. Boundary polynomial.} 
Fix an index $i_*$ with $a_{i_*,j}\approx\beta_j$. 
For $l\ne i_*$, the coefficient of $|\partial_lP_j(D\tilde u_j)|^2$ in \eqref{sourceenergy} is bounded from below by a positive constant.
It then follows from \eqref{sourceenergy} and \eqref{hessianbound} that
\[
\int_{{B_\rho}}|\partial_lP_j(D\tilde{u}_j)|^2\,dX \le C_M\Theta_j^2+C/M.
\]
Sending $j\to\infty$, $M\to\infty$ and using $D\tilde{u}_j\to D\tilde{u}$ almost everywhere, we obtain
\begin{equation}\label{transverse}
\partial_lP(D\tilde{u}(X))=0\quad\text{for a.e. }X\in B_\rho \quad(l\ne i_*,\ a_{i_*,j}\approx\beta_j).
\end{equation}

Since $\sigma_{l,j}\ge\sigma_{n,j}$ and $\sigma\mapsto\sigma/(1+\sigma)$ is increasing on $(0,\infty)$,
\[
a_{l,j}^2\prod_{i\ne l}(1+\sigma_{i,j})
=\frac{h_j\sigma_{l,j}}{1+\sigma_{l,j}}\prod_{i=1}^n(1+\sigma_{i,j})
\ge \frac{h_j\sigma_{n,j}}{1+\sigma_{n,j}}\prod_{i=1}^n(1+\sigma_{i,j})
=a_{n,j}^2\prod_{i<n}(1+\sigma_{i,j}).
\]
Since $|DF|\ge1$, we conclude by \eqref{energyidentities} that
\begin{equation}\label{p6.2:t1}
\sum_{l=1}^n\frac{\prod_{i\ne l}(1+\sigma_{i,j})}{\sigma_{1,j}} |\partial_lP_j(D\tilde{u}_j)|^2 \ge\frac{a_{n,j}^2}{\gamma_j^2\sigma_{1,j}} \prod_{i<n}(1+\sigma_{i,j}) \ge\frac{a_{n,j}^2a_{2,j}^2}{h_j\gamma_j^2}.
\end{equation}
Here the last inequality uses $n\ge3$ to retain both $\sigma_{1,j}$ and $\sigma_{2,j}$ in the product. By \eqref{hessianbound}, fix $M$ with $|\mathcal Q_j(M)|\ge|{B_\rho}|/2$. Since $a_{n,j}\ge ch_j$, \eqref{sourceenergy} and \eqref{p6.2:t1} imply
\begin{equation}\label{axisproduct}
\frac{a_{2,j}\sqrt{h_j}}{\gamma_j}\le C\Theta_j\to0.
\end{equation}
By \eqref{largestaxisdivergence} and the ordering of the $\sigma_{i,j}$, after passing to a subsequence there is $m\in\{1,\ldots,n\}$,
\[
\sigma_{i,j}\to\infty\ \text{for}\ i\le m \quad \text{and}\quad  \sigma_{i,j} \ \text{is bounded for} \ i>m.
\]
For $k>m$ and $l\ge2$, we have $a_{k,j}\le C\sqrt{h_j}$ and $a_{l,j}\le a_{2,j}$. 
The uniform $C^2$ bound for $F$ and \eqref{axisproduct} therefore imply, on bounded sets,
\begin{equation}\label{mixedderivatives}
|\partial_l\partial_kP_j| \le C\frac{a_{k,j}a_{l,j}}{\gamma_j} \le C\frac{a_{2,j}\sqrt{h_j}}{\gamma_j}\to0 \quad (k> m, \ l\ge2).
\end{equation}
Thus $\partial_kP$ depends only on $Y_1$ for each $k>m$, and
\begin{equation}\label{polynomialform}
P(Y)=p_0(Y_1,\ldots,Y_m) +\sum_{k=m+1}^np_k(Y_1)Y_k.
\end{equation}
Here $p_0(Y_1,\ldots,Y_m)=P(Y_1,\ldots,Y_m,0,\ldots,0)$ is a polynomial of degree at most $d_0$ with $p_0(0)=0$;
and $p_k$ is the polynomial defined by $\partial_kP(Y)=p_k(Y_1)$, of degree at most $d_0-1$.

\vskip5pt
\noindent \emph{Step 3. Graph estimate.} 
Let $b_j$ be the target ellipsoid center and take $\widehat T_j$ from \eqref{regularizedmatrix} at $h_j$. 
Set
\begin{align}\label{graphcoordinates}
z=x_j+b_j+\widehat T_jZ\quad \text{and}\quad Y_j(Z)=T_j^{-1}(y(z)-y_j),\quad  (Z\in B_\rho) 
\end{align}
and
\[
\widehat P_j(Z)=P_j(Y_j(Z)),\quad  (Z\in B_\rho).
\]
Applying \eqref{regularizedidentity} with $(y_0,x_0)=(y_j,x_j)$ and evaluating at the minimizer, we have
\begin{align}
v_1(z)-v_1(x_j+y_j)-x_j\cdot(z-x_j-y_j)
&=v(y(z))-v(y_j)-x_j\cdot(y(z)-y_j)
+\tfrac12|x(z)-x_j|^2\notag\\
&=h_j\tilde{v}_j(Y_j(Z))+\tfrac12|x(z)-x_j|^2.
\label{graphidentity}
\end{align}
Here $Y_j(Z)$ is defined in \eqref{graphcoordinates}.
The two terms on the right are nonnegative,
and by \eqref{graphinterior} and \eqref{graphcoordinates}, each term is bounded by $3h_j/2$. Therefore
\begin{equation}\label{graphbounds}
|x(z)-x_j|\le C\sqrt{h_j}\quad\text{and}\quad \tilde{v}_j(Y_j(Z))\le3/2.
\end{equation}
Estimate~\eqref{raybound} then bounds $Y_j(Z)$ uniformly, thus $|\widehat P_j(Z)|\le C$ on $B_\rho$. 
By \eqref{graphderivatives},
\[
|D_z(F\circ y)|^2=|KDF(y)|^2\le DF(y)^{\mathsf{T}}KDF(y).
\]
The derivative on the left is zero almost everywhere on $\{F\circ y=0\}$. 
On its complement, $|F\circ y|\le C\gamma_j$. 
Consequently \eqref{graphmeasure} controls the integral over the whole ball:
\begin{align*}
\widetilde\nu_\epsilon(G_{h_j}(x_j))
\ge c\gamma_j^{-1+\epsilon}\int_{G_{h_j}(x_j)}|D_z(F\circ y)|^2\,dz
=c\gamma_j^{1+\epsilon}|\det \widehat T_j|\int_{B_\rho}\sum_{l=1}^n\frac{|\partial_l\widehat P_j|^2}{(a_{l,j}')^2}\,dZ.
\end{align*}
Dividing by $h_j^\epsilon C_{\mathrm{cov},j}'$ and multiplying by \eqref{targetbound}, we obtain
\begin{align}
\frac{\widetilde\nu_\epsilon(G_{h_j}(x_j))}{h_j^\epsilon C_{\mathrm{cov},j}'}
\frac{\nu_0(\Omega_{h_j}^*(x_j))}{h_j^\epsilon C_{\mathrm{cov},j}}
&\ge c\left(\frac{\gamma_j}{h_j}\right)^{2\epsilon}
\frac{\kappa_jh_j|\det\widehat T_j|}{\beta_j C_{\mathrm{cov},j}'}
\int_{B_\rho}\sum_{l=1}^n\frac{|\partial_l\widehat P_j|^2}{(a_{l,j}')^2}\,dZ\notag\\
&=c\left(\frac{\gamma_j}{h_j}\right)^{2\epsilon}
\frac{\kappa_jh_j}{\beta_j\beta_j'}
\int_{B_\rho}\sum_{l=1}^n\prod_{i\ne l}(1+\sigma_{i,j})|\partial_l\widehat P_j|^2\,dZ,
\label{graphproduct}
\end{align}
where we used $C_{\mathrm{cov},j}'=\beta_j'h_j^{n-1}/|\det\widehat T_j|$ and $(a_{l,j}')^2=h_j(1+\sigma_{l,j})$.
By \eqref{largestaxisdivergence}
\[
\beta_j'/\beta_j=\sqrt{1+\sigma_{1,j}^{-1}}\to1.
\] 
This together with $\frac{h_j}{\beta_j\beta_j'}=\frac{1}{\sigma_{1,j}}\frac{\beta_j}{\beta_j'}$
implies the coefficient of $|\partial_l\widehat P_j|^2$ is, up to fixed factors,
\[
\left(\frac{\gamma_j}{h_j}\right)^{2\epsilon} \frac{\kappa_j}{\sigma_{1,j}}\prod_{i\ne l}(1+\sigma_{i,j}).
\]
Using $\kappa_j\ge c$ and $\gamma_j/h_j\ge c$ from \eqref{p6.2:e0}, together with \eqref{graphproduct} and \eqref{graphmassratio}, we obtain
\begin{equation}\label{graphenergy}
\int_{B_\rho}\sum_{l=1}^n \frac{\prod_{i\ne l}(1+\sigma_{i,j})}{\sigma_{1,j}} |\partial_l\widehat P_j|^2\,dZ\le C\Theta_j^2.
\end{equation}
In particular, $\partial_l\widehat P_j\to0$ in $L^2$ whenever $l\ne i_*$ and $a_{i_*,j}\approx\beta_j$.

Let $Y_j^0=T_j^{-1}(b_j-y_j)\to Y^0$ after a subsequence. 
We infer by \eqref{graphcoordinates}
\begin{equation}\label{graphcoordinateconvergence}
[Y_j(Z)]_i=[Y_j^0]_i+ \frac{a_{i,j}'}{a_{i,j}}Z_i -\frac{[x(z)-x_j]_i}{a_{i,j}} \to[Y^0]_i+Z_i\quad\text{for} \ i=1,\ldots,m,
\end{equation}
uniformly, using the displacement bound in \eqref{graphbounds}, $a_{i,j}'/a_{i,j}\to1$, and $\sqrt{h_j}/a_{i,j}\to0$. 
If all $p_k$ with $k>m$ vanish in \eqref{polynomialform}, the coordinate convergence~\eqref{graphcoordinateconvergence}, the uniform boundedness of $Y_j$ and \eqref{polynomiallimit} imply that $\widehat P_j\to P(Y^0+Z)$ uniformly on $B_\rho$.
Passing to weak derivatives in \eqref{graphenergy} then yields
\begin{equation}\label{graphdirections}
\partial_lP\equiv0\quad(l\ne i_*,\ a_{i_*,j}\approx\beta_j),
\end{equation}
provided that $p_k\equiv0$ for every $k>m$ in \eqref{polynomialform}.
Unlike \eqref{transverse}, this identity is global but requires the extra condition above.

\smallskip\noindent \emph{Step 4. Classification.} 
Suppose $a_{2,j}\ge c\beta_j$ along a subsequence. Then, by \eqref{axisproduct},
\[
\beta_j\sqrt{h_j}/\gamma_j\to0.
\]
For $k>m$, the argument for \eqref{mixedderivatives} therefore also applies with $l=1$. As a result, 
\begin{equation*}
p_k(Y_1)=\partial_kP(Y)\equiv \text{const.},\qquad k=m+1,\ldots,n.
\end{equation*}
Applying \eqref{transverse} with $i_*=1$ makes $p_k\equiv0$. 
Since now $i_*=1,2$, we further deduce by \eqref{graphdirections} that $DP(Y)\equiv0$.
This contradicts $P(0)=0$ and $P\not\equiv0$. 
Hence 
\[
\frac{a_{2,j}}{\beta_j}\to0.
\]
As in Step 2 of the proof of Proposition~\ref{planarclassification}, 
since $a_{i,j}\le a_{2,j}$ for $i\ge2$, we have $T_j/\beta_j\to e^*\otimes e$, with $|e|=|e^*|=1$.

If some $p_k$ with $k>m$ is nonzero, then \eqref{transverse} and \eqref{polynomialform} imply
$p_k(e\cdot D\tilde{u})=0$ almost everywhere. By continuity of the gradient on the differentiability set of a convex function, this identity holds at every differentiability point. The real roots of $p_k$ give alternative (2).

Otherwise all these $p_k$ vanish.
By virtue of \eqref{graphdirections} with $i_*=1$,
we find $\partial_lP\equiv0$ for every $l\ge2$. 
As in Step 3 of the proof of Proposition~\ref{planarclassification}, 
this yields $P(Y)=p(e\cdot Y)$, where $p(t)=P(te)$ is nonzero and has degree at most $d_0$, proving alternative (1).
\end{proof}

\begin{proof}[Proof of Proposition~\ref{boundarydimension}]
Fix $\delta>0$. 
Take $0<\alpha<\min\{1,\delta\}$, $0<\epsilon<\min\{1,\delta/2\}$ and $N=2/\alpha$. 
By Lemma~\ref{growth}, $\dim_{\mathrm H}(E\cap Z_\alpha)\le n-1+\alpha<n-1+\delta$. By Lemmas~\ref{growth} and~\ref{duallocalization}, the remaining set~$E\setminus Z_\alpha$ admits a countable cover by sets~$\mathcal P$ as in Definition~\ref{preparedset}. The section and Taylor estimates in Lemmas~\ref{sectionbounds} and~\ref{taylorlimit} are uniform on each such set.

On each subset, we split $\mathcal P=\mathcal P_+\cup\mathcal P_0$. 
The selection argument of Proposition \ref{positivemass} shows $\dim_{\mathrm H}\mathcal P_+\le n-1+2\epsilon<n-1+\delta$. Proposition~\ref{planarclassification} for $n=2$ and Proposition~\ref{classification} for $n\ge3$ verify the hypotheses of Proposition~\ref{refinement}. Proposition~\ref{zeromass} therefore gives $\dim_{\mathrm H}\mathcal P_0\le n-1$. 
Countable stability and then $\delta\downarrow0$ yield $\dim_{\mathrm H}E\le n-1$.
\end{proof}

\begin{proof}[Proof of Theorem~\ref{main}]
By \eqref{decomposition} and Lemma~\ref{regularset}, $\Sigma=A\cup E$ and $\dim_{\mathrm H}A\le n-1$. Proposition~\ref{boundarydimension} gives $\dim_{\mathrm H}E\le n-1$. The assertion follows.
\end{proof}

\section*{Acknowledgement}
The AI tool ChatGPT (OpenAI) was used in the
preparation of this manuscript. The authors have reviewed all AI-assisted
content and take full responsibility for the final manuscript.

\end{document}